\documentclass[12pt]{amsart}
\usepackage[scale=0.65, centering, headheight=14pt]{geometry}
\usepackage[latin1]{inputenc}
\usepackage[T1]{fontenc}
\usepackage{lmodern}
\usepackage{times}
\usepackage[english]{babel}
\usepackage{esint}
\usepackage{mathtools}

\usepackage{xcolor,amsmath,mathrsfs,amssymb,accents,graphicx,epstopdf, amsthm}
\usepackage{ bbold }
\usepackage{comment}
\usepackage{float}
\usepackage{subfig}
\usepackage{tikz}

\usepackage{textgreek}

\usepackage{tikz}
\usepackage{bbm}
\usepackage{amsmath,amssymb,amsfonts,amsthm}
\usepackage{mathtools,accents}
\usepackage{mathrsfs}
\usepackage{xfrac}
\usepackage{array} 
\usepackage{aliascnt}
\usepackage{booktabs} 
\usepackage{array} 

\usepackage{verbatim} 
\usepackage{subfig} 

\usepackage{mathrsfs, dsfont}
\usepackage{amssymb}
\usepackage{amsthm}
\usepackage{amsmath,amsfonts,amssymb,esint}
\usepackage{graphics,color}
\usepackage{enumerate}
\usepackage{mathtools,centernot}
\usepackage{tikz-cd}

\usepackage{microtype}
\usepackage{paralist} 
\usepackage{cases}
\usepackage[alphabetic,nobysame]{amsrefs}
\allowdisplaybreaks

\usepackage{braket}
\usepackage{bm}

\usepackage[citecolor=blue,colorlinks]{hyperref}
\addto\extrasenglish{}

\usepackage{enumerate}
\usepackage{xcolor}
\usepackage{aliascnt}

\numberwithin{equation}{section}

\newtheorem{theorem}{Theorem}[section]
\newtheorem{question}{Question}
\newtheorem{lemma}[theorem]{Lemma}
\newtheorem{definition}[theorem]{Definition}
\newtheorem{remark}[theorem]{Remark}
\newtheorem{proposition}[theorem]{Proposition}
\newtheorem{corollary}[theorem]{Corollary}
\newtheorem{example}[theorem]{Example}

\usepackage{enumitem,xparse}
\newlist{Claim}{description}{2}
\setlist[Claim]{labelindent=2em,leftmargin=*}
\newif\ifInsideClaim
\newcounter{claim}[theorem]
\newcounter{cclaim}[claim]

\let\originalqedsymbol\qedsymbol

\newcommand{\equivclass}[1]{%
	#1/{\sim}%
}

\newcommand{\Q}{\mathbb{Q}}
\newcommand{\N}{\mathbb{N}}
\newcommand{\R}{\mathbb{R}}

\newcommand{\Z}{\mathbb{Z}}
\newcommand{\T}{\mathbb{T}}
\newcommand{\Leb}[1]{{\mathscr L}^{#1}} 

\renewcommand{\P}{\mathbb{P}}

\newcommand{\<}{\langle}
\renewcommand{\>}{\rangle}
\renewcommand{\a}{\alpha}
\renewcommand{\b}{\beta}
\renewcommand{\d}{\delta}
\newcommand{\D}{\Delta}
\newcommand{\e}{\varepsilon}
\newcommand{\g}{\gamma}
\newcommand{\G}{\Gamma}

\newcommand{\s}{\sigma}

\renewcommand{\t}{\tau}
\renewcommand{\O}{\Omega}
\renewcommand{\o}{\omega}
\newcommand{\z}{\zeta}

\newcommand{\supp}{\operatorname{supp}}

\newcommand{\dist}{\rm dist}

\newcommand{\vo}{\vec{o}\@ifnextchar{^}{\,}{}}

 \newcommand{\bb}{{\mbox{\boldmath$b$}}}

 \newcommand{\tauV}{{\kern-3pt\tau}}

 \newcommand{\oVVVk}{\overline{\mbox{\boldmath$V$}}\kern-3pt}
 \newcommand{\tVVVk}{\tilde{\mbox{\boldmath$V$}}\kern-3pt}
\newcommand{\WW}{{\mbox{\boldmath$W$}}}

 \newcommand{\aalpha}{{\mbox{\boldmath$\alpha$}}}

 \newcommand{\nnu}{{\mbox{\boldmath$\nu$}}}

 \newcommand{\ssigma}{{\mbox{\boldmath$\sigma$}}}
 \newcommand{\rrho}{{\mbox{\boldmath$\rho$}}}

 \newcommand{\eeta}{{\mbox{\boldmath$\eta$}}}

	\title{On the existence of Markovian measures on continuous paths} 
\author [J. Pitcho]{Jules Pitcho}
\address{Jules Pitcho
	\hfill\break  
GSSI, Via Michele Iacobucci, 2, 67100, L'Aquila, Italy}
\email{jules.pitcho@gssi.it}

\begin{document}
	\maketitle

\begin{abstract}
	Let $\eeta$ be a positive Radon measure on the space of continuous paths from $\R$ into a locally compact Polish space $Y$, and assume that $\eeta$ admits an invariant measure. 
	We explicit conditions on $\eeta$ such that successive Markovianisation of $\eeta$ over any dense and countable set of times have all limit points (in the weak-star topology on Radon measures) satisfying the strong Markov property.  We show that if $Y$ is a locally compact Polish group, and $\eeta$ is left or right translation invariant, then $\eeta$ satisfies such conditions. 
	Our proof uses the Zermelo-Fraenkel and Axiom of Dependant Choice axiomatisation of set theory, in which countable products of compacts are compact. 
\end{abstract}
\section{Introduction}
\subsection{Motivation} It is known since the nineteenth century that spatially Lipschitz vector fields admit a unique flow. In the late twentieth century and the early twenty first century, vector fields with only one weak spatial derivative -- whether a locally integrable function or a Radon measure -- have been shown to admit an almost everywhere unique flow in \cite{DPL89,AmbBV,BianchiniBonicatto20} (see also \cite{CripDeL06:estimates} for a quantitative stability  theory, and \cite{BianchiniDeNitti} for a regularity theory). For rougher vector fields, superposition principles \cite{ambrosiocrippaedi,Smirnov94} guarantee existence of measures concentrated on integral curves called \emph{Lagrangian representations}. 

The structure theory of Lagrangian representations is however little understood. Results are known in the case vector fields singular at the initial time \cite{Pitcho_Int,Pitcho_SDE}, in the case of scalar conservation laws \cite{Bianchini_Marconi_Structure,Bianchini_Marconi_Concentration} or in the case of autonomous one-dimensional ODE's \cite{BressanMazzolaNguyen-Markovian23}. 

An interesting question is then whether Lagrangian representations can be memoryless, that is whether knowledge of the present state of a Lagrangian representation suffices to uniquely determine its future evolution, without any further information on its past. This property is expressed mathematically through a Markov property. 
Motivated by this question, we directly study measures on the space of continuous paths valued in a locally compact Polish space, and ask which such measures admit a modification satisfying a Markov property. 

\subsection{Our contribution in a nutshell} Aim of this paper is to construct a Markovian version of any measure on the space of continuous paths valued in a locally compact Polish space satisfying a regularity property, and admitting an invariant measure. We also show that translation invariant measures satisfy this regularity property. 

More precisely, we will define the Markov operator $M_t$ at a time $t\in\R$, which maps a measure $\eeta$ on the space of continuous paths to the unique measure on the space of continuous paths $M_t(\eeta)$, for which the distribution observed before time $t$ and after time $t$ coincide with that of $\eeta$, but for which the distribution before and after time $t\in \R$ are independant. $\eeta$ is then said to be \emph{Markovianised} at time $t$, if $M_t(\eeta)=\eeta$. We will further say that $\eeta$ satisfies the Markov property if it is Markovianised at every time $t\in \R$. We will then show that successively applying the Markov operator at a finite number of times gives a unique measure, which does not depend on the ordering of these times, thereby showing that Markovianisation on a finite set of times is well-defined. The Markov hull of $\eeta$ will be defined as the set of all measures reachable in the weak-star limit by successive Markovianisation of $\eeta$ over any countable and dense subset of $\R$.

 We will prove two results. Firstly, we will show that under a regularity hypothesis on $\eeta$, every element of the (non-empty) Markov hull of $\eeta$ satisfies the (strong) Markov property. Secondly, we will show that every translation invariant $\eeta$ satisfies the aforementioned regularity property whereby every element in its Markov hull satisfies the (strong) Markov property.

\bigskip

\subsection{Statement of results}
We will now briefly introduce the mathematical language to state our results. 
\subsubsection{The space of continuous paths}
Let $Y$ be a locally compact Polish space, let $\mu$ be a positive Radon measure on $Y$. Let $\G$ be a locally compact Polish subspace of the space of continuous paths from $\R$ to $Y$  endowed with the topology of uniform convergence on compacts such that:
\begin{itemize} 
\item for each compact time interval $I\subset\R$ and each compact set $K$ in $Y$, the set $\{\g\in \G:\g(t)=y\text{ for some $y\in K$ and some $t\in I$}\}$ is compact in $\G$;
\item for every $h\in \R$, defining the map $b_h(\g(\cdot))=\g(\cdot- h)$ from the set of continuous maps from $\R$ into $Y$, into the same set, we have $b_h(\G)=\G$. 
\end{itemize} 
See Section \ref{subsec_measure_space_cont_paths} for details. 
\begin{example}
	If $Y=\R^d$, we could take $\G$ to be the space of $L$-Lipschitz paths for any $L>0$. 
\end{example}
Let $e_t(\g)=\g(t)$ be the evaluation map at time $t\in \R$ from $\G$ to $Y$ and let $(e_t)_\#$ be the induced pushforward map from Radon measures on $\G$ to Radon measures on $Y$. For every interval $I\subset \R$, we let $\G\lfloor_I$ denote the set of restrictions $\g\lfloor_I$ of $\g$ to $I$ where $\g$ runs in $\G$, which we endow with the final topology with respect to the restriction map $\G\ni \g\longmapsto \g\lfloor_I\in \G\lfloor_I$, i.e. the finest topology for which the restriction map is continuous. 
\subsubsection{The $\mu$-invariant measure on continuous paths}
Let $\eeta$ be a positive Radon measure on $\G$ such that $(e_t)_\#\eeta=\mu$ for every $t\in \R$. 
\subsubsection{The Markovianisation map at a single time}
Let $\{\eta_{t,x}\}_{x\in Y}$ (a Borel family of Radon probability measures on $\G$)  be a disintegration of $\eeta$ with respect to $e_t$ and $\mu$ (see Section \ref{subsec_disint}). 
For every $t\in \R$, we define the Markovianisation map at time $t$ by
\begin{equation}
M_t(\eeta)=\int_{Y} \eta^{(-\infty,t]}_{t,x} \otimes_{(t,x)} \eta^{[t,+\infty)}_{t,x}\mu(dx),
\end{equation}
where $\eta^{(-\infty,t]}_{t,x} $ is the pushforward of $\eta_{t,x}$ through the restriction map $\G\ni \g\longmapsto \g\lfloor_{(-\infty,t]}\in \G\lfloor_{(-\infty,t]}$, and $\eta^{[t,+\infty)}_{t,x} $ is the pushforward of $\eta_{t,x}$ through the restriction map $\G\ni \g\longmapsto \g\lfloor_{[t,+\infty)}\in \G\lfloor_{[t,+\infty)}$. The tensor product $\eta^{(-\infty,t]}_{t,x} \otimes_{(t,x)} \eta^{[t,+\infty)}_{t,x}$ will be rigourously constructed in Section \ref{subsec_tensor_product_through_map} and Section \ref{subsect_def_markov}. 
\subsubsection{The Markov property} 
We will say that $\eeta$ satisfies the Markov property, if for every $t\in \R$, we have $M_t(\eeta)=\eeta$. 
\subsubsection{The Markovianisation map at several times}
Given a finite subset $F:=\{t_1,\dots t_n\}\subset \R$, we define $M_F(\eeta):=(M_{t_n}\circ \dots\circ M_{t_1})(\eeta)$ and we will show that $M_F$ does not depend on the ordering of $F$ in Section \ref{subsec_markov_several_time}.
\subsubsection{The Markov hull}
Let $D$ be a dense countable subset of $\R$. We will call the Markov $D$-hull of $\eeta$ the set of accumulation points -- in the weak-star topology on Radon measure on $\G$ -- of the set whose elements are $M_{D_N}(\eeta)$ where $N$ runs in $\N$ and $(D_N)_{N\in\N}$ runs in the class of sequences of subsets of $D$ such that $D_N$ has exactly $N$ elements and $D_{N}\subset D_{N+1}$ for every $N\in\N$. See Definition \ref{def_markov_D-hull}. 

We will then call \emph{Markov hull }of $\eeta$ the union of the Markov $D$-hulls of $\eeta$ over all dense and countable subsets $D$ of $\R$. See Definition \ref{def_markov_hull}. 

\subsubsection{The space of continuous disintegrations}
We further define the set $\mathcal{X}_\mu$ of  all families $\{\nu_x\}_{x\in \supp\mu}$ of Radon probability measures on $\G$ such that the map
 $\supp\mu\ni x\longmapsto \nu_{x}$ is continuous and there exists $t\in \R$ such that for every $x\in \supp\mu$ we have $\supp\nu_x\subset \{\g\in\G:\g(t)=x\}$. We endow $\mathcal{X}_\mu$ with the topology of uniform convergence on compacts of continuous maps from $\supp \mu$ to the space of Radon probability measures  on $\G$. 
 
 \subsubsection{The strong Markov property} 
 We will say that $\eeta$ satisfies the strong Markov property, if $\eeta$ satisfies the Markov property and for every $t\in \R$, there is disintegration $\{\eta_{t,x}\}_{x\in \supp\mu}$ of $\eeta$ with respect to $e_t$ and $\mu$, which belongs to $\mathcal{X}_\mu$, and is such that  $\eta^{(-\infty,t]}_{t,x}\otimes_{(t,x)}\eta^{[t,+\infty)}_{t,x}=\eta_{t,x}$ for every $(t,x)\in \R\times \supp\mu.$ See Definition \ref{def_strong_markov}. 
\subsubsection{The Markov regular condition}
Finally, we will say that $\eeta$ is Markov regular, if for every compact $I\subset \R$ there exists a compact set $K$ in $\mathcal{X}_\mu$ such that for every finite subset $F$ of $\R$, and every $t\in I$, there exists a disintegration of $M_F(\eeta)$ with respect to $e_t$ and $\mu,$ which belongs to $K$. See Definition \ref{def_markov_regular}. 
\subsubsection{Statement of theorems}
Let us state our first result in the context of this paragraph. 
\begin{theorem}\label{thm_main_strong_markov}
	Assume that $\eeta$ is Markov regular. Then the Markov hull of $\eeta$ is non-empty, and each element of the Markov hull of $\eeta$ satisfies the strong Markov property. 
\end{theorem}
From now, we further endow $Y$ with the structure of a locally compact Polish group. Let now $\mu$ be the unique -- up to a multiplicative constant - left (resp. right) Haar measure on $Y$. We assume that the left (resp. right) action of $Y$ on $\G$ defined by $Y\times \G\ni (y,\g)\longmapsto \{y\cdot \g(t)\}_{t\in\R}\in \G$ is transitive, and define translation maps $l_y:\G\ni x\longmapsto y\cdot x\in \G$ (resp. $r_y$) for every $y\in  Y$. We will say that $\eeta$ is left invariant (resp. right invariant), if for every $y\in Y$, we have $(l_y)_\#\eeta=\eeta$ (resp. $(r_y)_\#\eeta=\eeta$). 
\begin{example}
	If $Y=\R^d$ and $\G$ is the space of $L$-Lipschitz paths for some $L>0$, then the Haar measure $\mu$ is the Lebesgue measure and the action $Y\times \G\rightarrow \G$ is well-defined, and transitive. Since $\R^d$ is commutative, we do not distinguish between left and right Haar measures or actions.
\end{example}
Let us now state the second result. 
\begin{theorem}\label{them_main_transl_inv}
	Assume that $\eeta$ is left invariant or right invariant. Then $\eeta$ is Markov regular. 
\end{theorem}
As a direct corollary of the two theorems above, we have the following. 
\begin{corollary}\label{cor_inv}
Assume that $\eeta$ is left or right invariant. Then the Markov hull of $\eeta$ is non-empty, and each element of the Markov hull of $\eeta$ satisfies the strong Markov property. 
\end{corollary}

\subsubsection{Further questions} 
Let us ask the following questions. 
\begin{question}
	What is a necessary and sufficient condition on $\eeta$ such that its Markov hull consists of exactly one element?
\end{question}
\begin{question}
	What is a necessary and sufficient condition on $\eeta$ such that every element in its Markov hull satisfies the strong Markov property? 
\end{question}

\subsection{Physical interpretation} Let us propose the following dictionary to interpret Corollary \ref{cor_inv} in the language of physics. We translate $\mu$-invariant Radon measure on the space of continuous paths to realisation of incompressible Lagrangian turbulence; we translate translation invariant to homogenous; we translate Markov property to memoryless; and finally an element of the Markov hull is translated to a model.

 We then loosely rephrase Corollary \ref{cor_inv} as 
`every translation invariant and $\mu$-invariant Radon measure on the space of continuous paths has each element of its (non-empty) Markov hull satisfying the Markov property.' Following the dictionary proposed above, this translates to the following. 

\bigskip 

{\emph{Every realisation of incompressible, homogeneous Lagrangian turbulence admits a memoryless model.}}

\bigskip

\subsection{Outline of ideas}
We now present the main ideas to prove Theorem \ref{thm_main_strong_markov}.
\subsubsection{Uniform spaces} Our proof uses the theory uniform spaces introduced by Andr\'e Weil \cite{andre_weil_uniform} and subsequently developed by Nicolas Bourbaki \cite{bourbaki_topology}. 
Let us motivate the concept of a uniform spaces.

 A sequence of pairs of points in a metric space will be called uniformly close, if there is a finite upper bound to the distance between each pair of points. Uniform closeness of a sequence of pairs of points is a property invariant under uniform equivalence of the metric. It is however not expressible solely in terms of the topology structure induced by the metric -- that in terms of the family of open balls induced by the metric. In other words, sequences of pairs of points in a metric space which are uniformly close cannot be distinguished from those which are not solely by the topology induced by the metric. 

As a topology, a uniformity is expressed in a set-theoretic language, wherein contrary to a metric space structure, no reference is made to the real line. However, as an enrichement over a topology, a uniformity induced by a metric distinguishes sequences of pairs of points which are uniformly close from those which are not. In other words, a uniformity is sufficiently rich to express the property of uniform closeness of a sequence of pairs of points. Concepts such as uniform continuity or equicontinuity can thus be defined for maps from a topological space into a uniform space.
In particular,  Ascoli's compactness criterion for the uniform topology on real-valued continuous functions can be generalised to functions valued in uniform space. 
\bigskip
\subsubsection{Tensor product of measures through a map and Markovianisation}
To define the Markovianisation of a measure on the space of paths valued in a locally compact Polish space, we will first define the tensor product of two measures through a map, and give a characterising property of this object. Given a measure $\eta_{t,x}$ concentrated on a set of paths, which are all valued at $x\in Y$ for some time $t\in\R$, we will define its Markovianisation at time $t$ as the tensor product $\eta^{(-\infty,t]}_{t,x}\otimes_{(t,x)}\eta_{t,x}^{[t,+\infty)}$, which is the unique probability measure on the space of continuous, whose distribution before time $t$ coincides with that $\eeta$, whose disitribution after time $t$ coincides with that of $\eeta$, but -- contrary to $\eeta$ -- the distribution before and after time $t$ are independant. 

Given now a measure $\eeta$ on the space of paths, whose one-marginal at time $t\in \R$ is the positive Radon measure $\mu$ on $Y$, denoting by $\{\eta_{t,x}\}_{x\in Y}$ a disintegration of $\eeta$ with respect to $e_t$ and $\mu$, we then define its Markovianisation $M_t(\eeta)$ at time $t$ as the following integral
\begin{equation}
M_t(\eeta)=\int_Y \eta^{(-\infty,t]}_{t,x}\otimes_{(t,x)}\eta_{t,x}^{[t,+\infty)}\mu(dx).
\end{equation}
Whenever $\eeta=M_t(\eeta)$, we will say that $\eeta$ satisfies the Markov property at time $t\in\R$. 
We will then show that successive Markovianistion of $\eeta$ over a finite set of times does not depend on the ordering of these times, thereby showing that Markovianisation over a finite set of times is well-defined. To do so, we will have to prove properties of bilinearity and associativity of the tensor product of measures. We will then define the Markov hull of a $\mu$-invariant measure $\eeta$ to be the set of all accumulation points -- in the weak-star topology on Radon measures -- of sequences of successive Markovianisation of $\eeta$ over any countable and dense subset of $\R$. The Markov property however does not pass into the weak-star limit. 
\bigskip
\subsubsection{The space of disintegrations $\mathcal{Z}_\mu$ with respect to $\mu$ and the disintegration map}
We define the set of maps from $\supp\mu$ into probability measures on the space of paths, for which there exists $t\in\R$ such that for every $x\in \supp\mu$ the image of $x$ through such a map is concentrated on paths whose value at time $t$ is $x$. We then identity such maps which coincide $\mu$-a.e. and form the set of equivalence classes $\mathcal{Z}_\mu$. On this set, we will define integrals, which consist in averaging the values of a (equivalence class) map $\nnu\in \mathcal{Z}_\mu$ with respect to a class of nonnegative and continuous probability density on $\supp\mu$. We then endow $\mathcal{Z}_\mu$ with the initial uniformity with respect to all such integrals, which is the coarsest uniformity on $\mathcal{Z}_\mu$ making all these integrals uniformly continuous. 
We further consider the topological space $\mathcal{C}_c(\R;\mathcal{Z}_\mu)$ consisting of continuous maps from $\R$ into $\mathcal{Z}_\mu$ endowed with the topology of compact convergence, and observe that precompact sets are characterised by Ascoli's theorem.

We then define the disintegration map, mapping a $\mu$-invariant measure $\eeta$ on the space of continuous paths to a continuous map from $\R$ into $\mathcal{Z}_\mu$, which sends every $t\in\R$ to be the equivalence class of disintegrations (restricted to $\supp\mu$) of $\eeta$ with respect to $e_t$ and $\mu$. We will show that the disintegration map has precompact image in $\mathcal{C}_c(\R;\mathcal{Z}_\mu)$ using Tychonov's theorem for countable products (which can be proved in the Zermelo-Fraenkel and Axiom of Dependant Choice axiomatisation of set theory as shown in \cite[Proposition 4.72]{herrlich_axiom_of_choice}), thus enriching the information on the convergence to elements of the Markov hull. However this will still not be sufficient in order pass into the limit in the Markov property.
\bigskip 
\subsubsection{The space of continuous disintegrations $\mathcal{X}_\mu$ with respect to $\mu$}
We define the set $\mathcal{X}_\mu$ of continuous maps from $\supp\mu$ into probability measures on the space of paths (endowed with the vague uniformity inducing the weak-star topology), for which there exists $t\in\R$ such that for every $x\in \supp\mu$ the image of $x$ through such a map is concentrated on paths whose value at time $t$ is $x$. We endow this set with the uniform structure of compact convergence, and show that the canonical injection of $\mathcal{X}_\mu$ into $\mathcal{Z}_\mu$ is uniformly continuous. 

Assuming now that $\eeta$ is \emph{Markov regular} will ensure that disintegrations of $\eeta$ with respect to $\mu$ of successive Markovianisations of $\eeta$ are always contained in some compact subset of $\mathcal{X}_\mu$. Since on a compact space, there exists a unique uniformity compatible with the topology, we will deduce that the image through the disintegration map of any successive Markovianisations of $\eeta$ is precompact in the topological space $\mathcal{C}_c(\R;\mathcal{X}_\mu)$ of continuous maps from $\R$ to $\mathcal{X}_\mu$ endowed with the topology of compact convergence. 

Therefore, we will be able to upgrade the convergence to an element $\bar\eeta$ in the Markov hull: for every $t\in\R$, we will have that the unique continuous disintegration (restricted to $\supp\mu$) with respect to $e_t$ and $\mu$ of the successive Markovianisations of $\eeta$ converges in $\mathcal{X}_\mu$ to the unique continuous disintegration of $\bar\eeta$ with respect to $e_t$ and $\mu$. This convergence will be strong enough to pass into the limit in the Markov property, and to further propagate the Markov property to all times (not only the dense and countable set of times of successive Markovianisation). It will in fact be shown that $\bar\eeta$ satisfies the \emph{strong Markov property}. 
\bigskip
\subsubsection{Translation invariant measures}
We now consider the set of (left or right) translation invariant measures on a space of continuous paths valued in a locally compact Polish group. We will show that for any compact time interval $I$ in $\R$, and there exists a compact set $K$ in $\mathcal{X}_\mu$ such that for any translation invariant $\eeta,$  there exists a disintegration $\{\eta_{t,x}\}_{x\in\supp\mu}$ of $\eeta$ with respect to $e_t$ and $\mu$, which belongs to $K$. We will further show that the set of translation invariant measures is invariant by Markovianisation, thereby showing that all its elements are Markov regular. 
\subsection{Organisation of this paper}
In Section \ref{sec_uniform}, we collect definitions and facts on uniform structures and the induced topology. Firstly we will give basic definitions and facts. Secondly we will present the key concept of an initial uniformity, and the uniformity induced by a family of pseudometrics. Thirdly we will state that on a compact space, there exists exactly one uniformity compatible with the topology. This result will prove to be essential.  Fourthly, we will show how to endow a function space with different uniformities such as the uniformity of uniform convergence, the uniformity of $\mathfrak{S}$-convergence, or the uniformity of compact convergence. Finally, we will restrict these uniformities to spaces of continuous functions and state a generalisation of Ascoli's compact criterion. 

In Section \ref{sec_measure_theory_top_space}, we gather some facts on Radon measures, and on the uniformity of vague convergence on Radon measures (inducing the weak-star topology on Radon measures). We will also present the disintegration of a measure, which will be useful. We will further present the space of paths $\G$ over which we will work and show that it is locally compact and Polish. We will then introduce the space $\mathcal{M}_\mu(\G)$ of $\mu$-invariant Radon measures on $\G$, and we will show that the vague uniformity on $\mathcal{M}_\mu(\G)$ is metrizable and compact. Finally we will describe the uniformity of vague convergence on Radon probability measures on $\G$.  

In Section \ref{sec_tensor_prod_measure}, we present the tensor product of measures. In a general measure-theoretic context, we will construct the tensor product of two measures through a map, of which we will give a characterising property. We will then show that this tensor product is well-behaved with respect to composition of maps. 

In Section \ref{sec_markovianisation}, using the tensor product presented in the the previous section, we explain the process of Markovianisation at some time $t\in\R$ of a measure in $\mathcal{M}_\mu(\G)$ to another measure in $\mathcal{M}_\mu(\G)$, and we show that successive Markovianisation gives a measure in $\mathcal{M}_\mu(\G)$, which does not depend on the ordering of times of Markovianisation. In doing so, we will prove properties of the tensor product of measures such as bilinearity and associativity. 

In Section \ref{sec_disint_map}, we introduce the spaces of disintegrations $\mathcal{X}_\mu$ and $\mathcal{Z}_\mu$ with respect $\mu$, and introduce the disintegration map from a measure in $\mathcal{M}_\mu(\G)$ to a continuous map from $\R$ to $\mathcal{Z}_\mu$. We will collect useful property of this map, study the relation between $\mathcal{X}_\mu$ and $\mathcal{Z}_\mu$, and show that the Markov property can propagate to limits in a simplified setting. Upon expliciting the definitions of Markov regularity and Markov hull, we will then be able to conclude the proof of Theorem \ref{thm_main_strong_markov}. 

In Section \ref{sec_translation_inv}, we will give the proof of Theorem \ref{them_main_transl_inv}. 
\subsection*{Acknowledgements} 
I thank the Gran Sasso Science Institute for providing excellent working conditions. I thank Stefano Bianchini for discussions from which this article originates. 
	\section{Uniform structures and topology}\label{sec_uniform}
	In this section, we state definitions and basic properties of uniform structures without proofs. A detailed presentation can be found in \cite{bourbaki_topology}. 
\subsection{Neighbourhoods}
We recall the notion of neighborhood of a point from \cite[Chapter I]{bourbaki_topology}, and that a system of neighborhoods of every point uniquely specifies a topology. 
\begin{definition}
	Let $X$ be a topological space and $A$ any subset of $X$. A neighbourhood of $A$ is any subset of $X$ which contains an open set containing $A$. The neighborhoods of a subset $\{x\}$ consisting of a single point are also called neighbourhoods of the point $x$. 
\end{definition}
Let us denote by $\mathfrak{B}(x)$ the set of all neighbourhoods of $x$. The set $\mathfrak{B}(x)$ have the following properties:
\begin{itemize}
	\item [$(V_1)$] Every subset of $X$ which contains a set belonging to $\mathfrak{B}(x)$ itself belongs to $\mathfrak{B}(x).$
	\item [$(V_2)$] Every finite intersection of sets of $\mathfrak{B}(x)$ belongs to $\mathfrak{B}(x).$
	\item [$(V_3)$] The element $x$ is in every set of $\mathfrak{B}(x)$. 
	\item [$(V_4)$] If $V$ belongs ot $\mathfrak{B}(x)$, then there exists a set $W$ belonging to $\mathfrak{B}(x)$ such that for every $y\in W$, $V$ belongs to $\mathfrak{B}(y)$. 
\end{itemize}
\begin{proposition}
	If to each element $x$ of a set $X$, there corresponds a set $\mathfrak{B}(x)$ of subsets of $X$ such that the properties $(V_1), (V_2), (V_3)$ and $(V_4)$ are satisfied, then there is a unique topological structure on $X$ such that, for every $x\in X$, $\mathfrak{B}(x)$ is the set of neighbourhoods of $x$ in this topology. 
\end{proposition}
\subsection{Uniformities} 
We give definitions and elementary properties of uniform structures as in \cite[Chapter II]{bourbaki_topology}. First let us recall the definition of a filter. 
\begin{definition} 
	A filter on a set $X$ is a family of subsets of $X$  satisfying the following. 
	\begin{itemize}
		\item [($F_1)$] Every subset of $X$ which contains a set of $\mathcal{F}$ belongs to $\mathcal{F}$.
		\item [($F_2)$] Every finite intersection of sets of $\mathcal{F}$ belongs to $\mathcal{F}$.
		\item [$(F_3)$] The empty set is not in $\mathcal{F}$. 
	\end{itemize}
\end{definition} 
\begin{remark}
	In a topological space space $X$, the set of all neighbourhoods of an arbitrary non-empty subset $A$ of $X$ (and in particular the set of all neighbourhoods of a point of $x$) is a filter, called the neighbourhood filter of $A.$
\end{remark}
We can now give the definition of a uniform structure. 
\begin{definition} \label{def_uniform_structure}
	A uniform structure (or uniformity) on a set $X$ is a structure given by a family of subsets $\mathcal{U}$ of $X\times X$ which satisfy $(F_1)$ and $(F_2)$ and the following axioms:
	\begin{itemize}
		\item [$(U_1)$] Every set belonging to $\mathcal{U}$ contains the diagonal $\D$.
		\item [$(U_2)$] If $V\in \mathcal{U}$ then $\overset{- 1}{V}\in \mathcal{U}$. 
		\item[$(U_3)$] For every $V\in \mathcal{U}$ there exists $W\in \mathcal{U}$ such that $W\circ W\subset V.$
	\end{itemize}
\end{definition} 
The sets in $\mathcal{U}$ are called \emph{entourages} of the uniformity defined on $X$. 
\begin{remark}
	The composite of two subsets $U$ and $V$ of $X\times X$ is defined by 
	\begin{equation}
	U\circ V =\{ (x,z)\;:\; \text{there exists $y\in X$ such that $(x,y)\in U$ and $(y,z)\in V$}\}. 
	\end{equation}
	The inverse of $V$ is defined by 
	\begin{equation}
	\overset{- 1}{V}=\{(y,x)\in X\times X\;:\; (x,y)\in V\}. 
	\end{equation}
\end{remark}
A uniformity can be uniquely determined by a subfamily of entourages. 
\begin{definition}
	A fundamental system of entourages of a uniformity is any set $\mathfrak{B}$ of entourages such that every entourage contains a set belonging to $\mathfrak{B}.$
\end{definition}

Given a subset $V$ of $X\times X$ and $x\in X$, we denote by $V(x)$ the set of all $y\in X$ such that $(x,y)\in V$. 
\begin{proposition}
	Let $X$ be a set endowed with a uniform structure $\mathcal{U}$, and for every $x\in X$, let $\mathfrak{B}(x)$ be the set of subsets of $V(x)$ of $X$, where $V$ runs through the set of entourages of $\mathcal{U}$. Then there exists a unique topology on $X$ such that, for every $x\in X$, $\mathfrak{B}(x)$ is the neighbourhood filter of $x$ in this topology. 
\end{proposition}
This is topology is called \emph{the topology induced by the uniformity} $\mathcal{U}$. 
Finally let us state a criterion for a uniform space to be Hausdorff. 
\begin{theorem}
	A uniform space $X$ is Hausdorff if and only if the intersection of all the entourages of its uniform structure is the diagonal $\D$ of $X\times X$.
\end{theorem}
\subsection{Uniformly continuous functions} 
We now define morphisms between uniform spaces. These are uniformly continuous functions.
\begin{definition}
	A mapping $f$ of a uniform space $X$ into a uniform space $X'$ is said to be uniformly continuous, if for every entourage $V'$of $X'$, there exists an entourage $V$ of $X$ such that $(x,y)\in V$ implies $(f(x),f(y))\in V'$. 
\end{definition}
Note that every uniformly continuous mapping is continuous. 
Let us now define isomorphism of uniform spaces. 
\begin{definition}
	A mapping $f$ of a uniform space $X$ into a uniform space $X'$ is said to be an isomorphism of uniform spaces, if $f$ is bijective, and both $f$ and $f^{-1}$ are uniformly continuous. 
\end{definition}
\subsection{Comparison of uniformities} 
The class of uniform structures on a given set is a partially ordered set. Let us define the order relation. 
\begin{definition}
	If $\mathcal{U}_1$ and $\mathcal{U}_2$ are two uniform structures on the same set $X$, $\mathcal{U}_1$ is said to be finer than $\mathcal{U}_2$ if, denoting by $X_i$ the set $X$ with the uniform structure $\mathcal{U}_i$ ($i=1,2$), the identity mapping $X_1\to X_2$ is uniformly continuous. 
\end{definition}
The following proposition asserts that a finer uniformity contain more entourages. 
\begin{proposition}
	If $\mathcal{U}_1$ and $\mathcal{U}_2$ are two uniformities on a set $X$, then $\mathcal{U}_1$ is finer than $\mathcal{U}_2$, if and only if every entourage $\mathcal{U}_2$ is an entourage of $\mathcal{U}_1$. 
\end{proposition}
The following shows that the corresponding order relation induced between topologies coincides with the usual order relation between topologies. 
\begin{corollary}
	Let $\mathcal{U}_1$ and $\mathcal{U}_2$ be two uniformities on a set $X$, and suppose that $\mathcal{U}_1$ is finer than $\mathcal{U}_2$; then the topology induced by $\mathcal{U}_1$ is finer than the topology induced by $\mathcal{U}_2$.  
\end{corollary}
\subsection{Initial uniformities} 
Let us now define the initial uniformity on a set with respect to a family maps. This will be concept will be very useful. 
\begin{proposition} \label{prop_intial}
	Let $X$ be a set, let $\{Y_i\}_{i\in I}$ be a family of uniform spaces, and for every $i\in I$, let $f_i$ be a mapping from $X$ to $Y_i$. For each $i\in I$, let $g_i$ denote $f_i\times f_i$. Let $\Theta$ be the family of subsets of $X\times X$ of the form $g^{-1}_i(V_i)$, where $i\in I$ and $V_i$ is an entourage of $Y_i$, and let $\mathfrak{B}$ be the set of all finite intersections 
	\begin{equation}
g_1^{-1}(V_{1})\cap \dots \cap g^{-1}_{n}(V_{n}).
	\end{equation}
	of $\Theta.$ Then $\mathfrak{B}$ is a fundamental system of entourages of a uniformity $\mathcal{U}$ on $X$ which is the initial uniform structure on $X$ with respect to the family $\{f_i\}_{i\in I}$, and in particular $\mathcal{U}$ is the coarsest uniformity on $X$ for which all mappings $f_i$ are uniformly continuous. Furthermore, let $h$ be a mapping of a uniform space $Z$ into $X$; then $h$ is uniformly continuous if and only if the mappings $f_i\circ h$ are uniformly continous. 
\end{proposition}
The dual notion of final uniformity can also be defined and satsifies corresponding dual properties. 
\subsection{Uniformity induced by a family of pseudometrics}\label{subsec_uniform_pseudometrics}
Let us begin by giving the definition of a uniformity induced by a pseudometric.
\begin{definition}
	Given a pseudometric ${\rm d}$ on a set $X$, the uniformity defined by ${\rm d}$ on $X$ which has as a fundamental system of entourages the family of sets ${\rm d}^{-1}([0,\a])$, where $\a$ runs through the set of all positive real numbers. 
\end{definition}
We can now define the uniformity induced by a family of pseudometrics.
\begin{definition}
	If $\{{\rm d}_i\}_{i\in I}$ is a family of pseudometrics on a set $X$, then the least upper bound of the set of uniformities defined on $X$ by the pseudometrics ${\rm d}_i$ is called the uniformity defined by the family $\{{\rm d}_i\}_{i\in I}$. 
\end{definition}
A fundamental system of entourages of the uniformity defined by a family of pseudometrics $\{{\rm d}_i\}_{i\in I}$ is given by 
\begin{equation}
\bigcap_{i\in J}{\rm d}^{-1}_{i}([0,\a])
\end{equation}
where $\a$ runs through the positive rational numbers, and $J$ runs through the finite subsets of $I$. 
\subsection{Uniformity on a compact space} 
The following existence and uniqueness theorem will be fundamental to the the proof of Theorem \ref{them_main_transl_inv}.  
\begin{theorem}\label{thm_unique_uniform_structure_compact}
	On a compact space $X$, there exists exactly one uniformity compatible with the topology of $X$; the entourages of this uniformity are all the neighborhoods of the diagonal of $\D$ in $X\times X$. Furthermore, $X$ endowed with this uniformity is a complete uniform space. 
\end{theorem}
\subsection{The uniformity of uniform convergence on function spaces} 
Given two sets $X$ and $Y$, we define the set of all maps from $X$ to $Y$ as follows $$\mathcal{F}(X;Y):=\{f:X\to Y\}.$$ 
Assume further now that $Y$ is a uniform space. For every entourage $V$ of $Y$, we let ${\mathbf W}(V)$ denote the set of all pairs $(u,v)$ of mappings of $X$ into $Y$ such that $(u(x),v(x))\in V$ for every $x\in X$. As $V$ runs through the set of entourage of $Y$, the sets ${\mathbf W}(V)$ form a fundamental system of entourages of a uniformity of $\mathcal{F}_u(X;Y)$ on $\mathcal{F}(X;Y)$. Let us define the uniformity of uniform convergence. 
\begin{definition}
	The uniformity on the set $\mathcal{F}_u(X;Y),$ which has as a fundamental system of entourages the set of subsets ${\mathbf W}(V)$, where $V$ runs through the set of entourages of $Y$, is called the uniformity of uniform convergence. The topology it induces is called the topology of uniform convergence. 
\end{definition}

\subsection{$\mathfrak{S}$-convergence}
The uniformity $\mathfrak{S}$-convergence refines the uniformity of uniform convergence. 
\begin{definition}
	Let $X$ be a set, $Y$ a uniform space, $\mathfrak{S}$ a set of subsets of $X$. The uniformity of uniform convergence in the sets of $\mathfrak{S}$, or simply the uniformity of $\mathfrak{S}$-convergence, is the coarsest uniformity on $\mathcal{F}(X;Y)$ which makes uniformly continuous the restriction mappings $ u\longmapsto u\lfloor_A$ of $\mathcal{F}(X;Y)$ into $\mathcal{F}_u(A;Y)$ where $A$ runs through $\mathfrak{S}$. The uniform space obtained by endowing $\mathcal{F}(X;Y)$ with the uniformity of $\mathfrak{S}$-convergence is denoted by $\mathcal{F}_{\mathfrak{S}}(X;Y).$
\end{definition}
A fundamental system of entourages of $\mathcal{F}_\mathfrak{S}(X;Y)$ may be obtained as follows: for each $A\in \mathfrak{S}$ and each entourage $V$ of a fundamental system of entourages $\mathfrak{B}$ of $Y$, let $\WW(A;V)$ be the set of mappings $(u,v)$ of $X$ into $Y$ such that $(u(x),v(x))\in V$ for each $x\in A$; the \emph{finite intersection of the sets} $\WW(A;V)$ where $A$ runs through $\mathfrak{S}$ and $V$ runs through $\mathfrak{B}$ form a fundamental system of entourages of $\mathcal{F}_\mathfrak{S}(X;Y).$ It is thus immediate that if $\mathfrak{S}$, $\mathfrak{S}'$ are two sets of subsets of $X$ such that $\mathfrak{S}\subset \mathfrak{S}'$, then the uniformity of $\mathfrak{S}'$-convergence is finer than the uniformity of $\mathfrak{S}$-convergence. 

\subsection{Examples of $\mathfrak{S}$-convergence }
We now give three examples of $\mathfrak{S}$-convergence.
\subsubsection{Uniform convergence in a subset of $X$}
Let $A$ be a subset of $X$ and take $\mathfrak{S}=\{A\}$. The uniformity of $\mathfrak{S}$-convergence is then called the uniformity of uniform convergence in $A$.

\subsubsection{Pointwise convergence in a subset of $X$} Let $A$ be a subset of $X$, and take $\mathfrak{S}$ to be the set of all subsets of $X$, which consist of a single point belonging ot $A$. The uniformity of $\mathfrak{S}$ is the called the uniformity of pointwise convergence in $A$. 

\subsubsection{Compact convergence} Suppose that $X$ is a topological space, and take $\mathfrak{S}$ to be the set of all compacts of $X$. The uniformity of $\mathfrak{S}$-convergence is then called the uniformity of compact convergence, and the uniform space obtained by endowing $\mathcal{F}(X;Y)$ with this uniformity is denoted by $\mathcal{F}_c(X;Y)$. 

The uniform structure of compact convergence is coarser than that of uniform convergence, and the two coincide if $X$ is compact; it is also finer than the uniform structure of pointwise convergence, and these two coincide if $X$ is discrete. 

\subsection{$\mathfrak{S}$-convergence in spaces of continuous maps}
Let now $X$ be a topological space and let $C(X;Y)$ denote the set of all \emph{continuous mappings} of $X$ into $Y$. If $\mathfrak{S}$ is a set of subsets of $X$, and if $Y$ is a uniform space, we denote by $\mathcal{C}_{\mathfrak{S}}(X;Y)$ the uniform space on ${C}(X;Y)$ endowed with the uniformity of $\mathfrak{S}$-convergence, i.e. the initial uniformity with respect to the canonical injection $C(X;Y)\hookrightarrow \mathcal{F}_\mathfrak{S}(X;Y)$. 

\begin{theorem}\label{theorem_hausdorff_complete}
	Let $X$ be a topological space, let $\mathfrak{S}$ be a set of subsets of $X$ and let $Y$ be a uniform space. The following hold:
	\begin{enumerate}
\item if $Y$ is Hausdorff and if the union of the sets $\mathfrak{S}$ is dense in $X$, then $\mathcal{C}_{\mathfrak{S}}(X;Y)$ is Hausdorff;
\item if $Y$ is a complete, then $\mathcal{C}_{\mathfrak{S}}(X;Y)$ is complete. 
	\end{enumerate} 
\end{theorem}
We also have the following for the uniformity of compact convergence in the spaces of continuous maps. 
\begin{proposition}
	If $X$ is either metrizable or locally compact, and $Y$ is a uniform space. Then $\mathcal{C}_c(X;Y)$ is closed in the uniform space $\mathcal{F}_c(X;Y)$; if in addition $Y$ is complete, then the uniform space $\mathcal{C}_c(X;Y)$ is complete. 
\end{proposition} 
Let us now define equicontinuity which we will use to state a compactness criterion in $\mathcal{C}_{\mathfrak{S}}(X;Y)$. 
\begin{definition}
	Let $X$ be a topological space and $Y$ a uniform space. A subset $H$ of $\mathcal{F}(X;Y)$ is said to be equicontinuous at a point $x_0\in X$, if for each entourage $V$ of $Y$, there exists a neighborhood $U$ of $x_0$ in $X$ such that $(f(x_0),f(x))\in V$ for every $x\in U$ and every $f\in H$. $H$ is said to be equicontinuous if it is equicontinuous at every point in $X$. 
\end{definition}
We now state a compactness criterion in $\mathcal{C}_{\mathfrak{S}}(X;Y)$, which will be essential in proving Theorem \ref{thm_main_strong_markov}. 
\begin{theorem}[Ascoli]
	Let $X$ be a topological space, let $\mathfrak{S}$ be a covering of $X$ by compacts, let $Y$ be a uniform space and $H$ a set of continuous mappings of $X$ into $Y$ such that for each $A\in \mathfrak{S}$ and each $u\in H$, the restriction $u$ to $A$ is continuous. Then for $H$ to be precompact with respect to the uniformity on $\mathcal{C}_{\mathfrak{S}}(X;Y)$, it is necessary and sufficient that the following hold:
	\begin{enumerate}
		\item  for each $A\in \mathfrak{S}$, the set $H \lfloor_A\subset \mathcal{F}(A;Y)$ of restrictions to $A$ of functions of $H$ is equicontinuous;
		\item for each $x\in X$, the set $H(x)\subset Y$ of points $u(x)$ is precompact. 
	\end{enumerate}
\end{theorem}
Let us finally quote the following theorem from \cite[Chapter X, 3.3]{bourbaki_topology}.
\begin{theorem}\label{thm_separable}
	Let $X$ be a locally compact space whose topology admits a countable basis, and let $Y$ be a metrizable separable uniform space. Then the space $\mathcal{C}_c(X;Y)$ of continuous mappings of $X$ into $Y$, endowed with the topology of compact convergence, is a metrizable separable space. 
\end{theorem}

\section{Measure theory on topological spaces} \label{sec_measure_theory_top_space}
In this section, we first present some useful results of measures theory of topological spaces. A detailed presentation is given in Laurent Schwartz's book \cite{schwartz_radon_measures}. We then specify to measures where the underlying topological space is a space of continuous paths valued in a locally compact Polish space. 
\subsection{Radon measures on topological spaces}
Let $X$ be a locally compact Polish space and let $\mathscr{B}(X)$ be its Borel $\s$-algebra. Let $C_c(X)$ be the space of compactly supported, real-valued continuous functions on $X$. It is a Banach space under the $\sup$ norm. A functional on $C_c(X)$ is a map from $C_c(X)$ to the real numbers. A functional $\lambda$ is said to be \emph{positive}, if $\Phi\geq 0$ implies that $\lambda(\Phi)\geq 0$. 

\subsubsection{Basic facts on Radon measures}
By convention, a Radon measure will always be assumed to be positive. 
Let us begin with the definition of Radon measures as a real-valued set map. 
\begin{definition}
	A map $\nu:\mathscr{B}(X) \to \R^+$ is said to be a Radon measure on $X$, if 
	\begin{enumerate}
		\item $\nu$ is outer regular on all Borel sets, i.e.
		\begin{equation}
		\nu(A)=\inf\{\nu(B)\;:\;B \text{ open and } A\subset B\};
		\end{equation}
		\item $\nu$ is inner regular on all Borel sets, i.e. 
		\begin{equation}
		\nu(A)=\sup\{\nu(F)\;:\; F\subset A, \text{ $F$ compact and measurable}\};
		\end{equation}
		\item $\nu$ is finite on compacts, i.e. $\nu(K)<+\infty$ for every $K$ compact in $X$. 
	\end{enumerate}
	
\end{definition}
Let $\mathcal{M}(X)$ be the set of Radon measures on $X$. 
Every Radon measure on $X$ induces a continous linear functional on $C_c(X)$ through the formula
\begin{equation}\label{eqn_riesz_mark}
\lambda(\Phi)=\int_X \Phi(x)\nu(dx)\qquad\forall\Phi\in C_c(X). 
\end{equation}
Conversely, by the Riesz-Markov-Kakutani theorem, every Radon measure arises as a positive continuous linear functional on $C_c (X)$ through the formula \eqref{eqn_riesz_mark}.
\subsubsection{Vague uniformity on Radon measures}  We henceforth equip the set $\mathcal{M}(X)$ with the uniform structure induced by the family of pseudometrics
\begin{equation}
{\rm d}_\Phi(\nu^1,\nu^2)=\Big|\int_X \Phi(x)\nu^1(dx)-\int_X \Phi(x)\nu^2(dx)\Big|,
\end{equation}
where $\Phi$ runs in $C_c(X)$. This uniform structure will be called the \emph{vague uniformity}, and the induced topology will be called the \emph{vague topology}. A fundamental system of entourage of the vague uniformity on $\mathcal{M}(X)$ is given by $\bigcap_{ \Phi\in J}{\rm d}^{-1}_{\Phi}([0,\a])$ where $J$ runs through all the finite subsets of $C_c(\G)$, and $\a$ runs in the nonnegative real numbers. 

Any subset $B$ of $\mathcal{M}(X)$ equipped with the uniformity induced from the vague uniformity will also be said to be equipped with the vague uniformity. A fundamental system of entourage of the vague uniformity on $B$ is then given by $\bigcap_{ \Phi\in J}{\rm d}^{-1}_{\Phi}([0,\a])\cap B\times B$ where $J$ runs through all the finite subsets of $C_c(\G)$. 
\subsubsection{Separability and metrizability}
 By Thereom \ref{thm_separable}, there exists a dense and countable subset $\mathscr{D}\subset C_c(X)$. 
\begin{lemma} \label{lem_separability}
Let $M>0$. Then the vague uniformity on $\{\nu\in \mathcal{M}(X)\;:\;\nu(X)=M\}$ has a  countable fundamental system of entourages, and is metrizable.  
\end{lemma}
\begin{proof}
By \cite[Chapter IX, 2.4]{bourbaki_topology}, it suffices to show that the family of pseudometrics $\{{\rm d}_\Phi\}_{\Phi\in \mathscr{D}}$ induces a uniformity on $\{\nu\in \mathcal{M}(X)\;:\;\nu(X)=M\}$, which is isomorphic to the vague uniformity of $\{\nu\in \mathcal{M}(X)\;:\;\nu(X)=M\}$. 
	By the description of the fundamental system of entourages of $\mathcal{M}(X)$ above, we will thus have to show that for every $\a\in\R^+$ and every $\Phi\in C_c(\G)$, there exists $\b\in \Q^+$ and $\bar\Phi\in \mathscr{D}$ such that
\begin{equation}
{\rm d}^{-1}_{\bar\Phi}([0,\b])\subset {\rm d}^{-1}_{\Phi}([0,\a]). 
\end{equation}
Let $\Phi\in C_c(\G)$ and $\a\in \R^+$. Choose $\bar\Phi\in \mathscr{D}$ such that $\sup_{x\in X}|\bar\Phi(x)-\Phi(x)|\leq \a/4M$ and $\b\in (0,\a/2)\cap \Q$. 
Then we have 
\begin{equation}
\begin{split} 
&\Big|\int_X\Phi(x)\nu^1(dx)-\int_X\Phi(x)\nu^2(dx)\Big|\\
&\leq \int_X |\Phi(x)-\bar\Phi(x)|\nu^1(dx)+\int_X |\Phi(x)-\bar\Phi(x)|\nu^2(dx)\\
&\qquad+\Big|\int_X\bar\Phi(x)\nu^1(dx)-\int_X\bar \Phi(x)\nu^2(dx)\Big|\\
&\leq \a/2 +\Big|\int_X\bar\Phi(x)\nu^1(dx)-\int_X\bar \Phi(x)\nu^2(dx)\Big|. 
\end{split} 
\end{equation}
Therefore, if $(\nu^1,\nu^2)\in {\rm d}_{\bar\Phi}^{-1}([0,\b])$, then  $(\nu^1,\nu^2)\in {\rm d}_{\Phi}^{-1}([0,\a])$, which proves the thesis. 
\end{proof}
We describe a countable fundamental system of entourages of the vague uniformity of $\{\nu\in \mathcal{M}(X)\;:\;\nu(X)=M\}$ by
\begin{equation}
\bigcap_{i=1}^n{\rm d}_{\Phi_i}^{-1}([0,\a]),
\end{equation}
where $n$ runs in $\N$, $\a$ runs in $\Q^+$ and $\Phi_i$ run in $\mathscr{D}$. Denoting by $\iota:\{\nu\in \mathcal{M}(X)\;:\;\nu(X)=M\}\hookrightarrow\mathcal{M}(X)$ the canonical injection, a metric inducing the vague uniformity is then given by $${\rm d}:=\sum_{k\in\N} {\rm d}_{\Phi_k}\circ (\iota\times\iota)/2^k.$$
The vague topology is thus characterised by the \emph{vague convergence} of sequences: we shall say that $(\nu_n)_{n\in\N}$ converges vaguely to $\nu$ as $n\to+\infty$, if 
\begin{equation}
\lim_{n\to+\infty}\int_X\Phi(x)\nu_n(dx)=\int_X\Phi(x)\nu(dx)\qquad\forall \Phi\in C_c(X).
\end{equation}

Recall that a topological space is called \emph{$\s$-compact}, if it admits a countable exhaustion by compacts. Note that a locally compact Polish space is always $\s$-compact.
The following lemma shows that metrizability holds for a larger class of subsets of subsets of $X$.
\begin{lemma}\label{lem_vague_unif_metrizable_compact_exhaus}
	Let $(A_k)_{k\in\N}$ be a countable exhaustion by compacts of $X$, and $(M_k)_{k\in\N}$ a sequence of positive real numbers. Then the vague uniformity on $$B:=\bigcap_{k\in\N}\{\nu\in \mathcal{M}(X)\;:\;\nu(A_k)=M_k\}$$ is metrizable. 
\end{lemma}
\begin{proof}
	The case $B=\emptyset$ is trivial. 
	For each $k\in\N$, define the map 
	\begin{equation}
	r_k:\mathcal{M}(X)\ni \nu\longmapsto \nu\lfloor_{A_k} \in \mathcal{M}(A_k).
	\end{equation}
	In view of Lemma \ref{lem_separability}, the vague uniformity on $\{\nu\in \mathcal{M}(A_k)\;:\;\nu(A_k)=M_k\}$ is metrizable, on which we let ${\rm d}_k$ be a metric. We now claim that ${ \rm d}:=\sum_{k\in\N} {\rm d}_k\circ( r_k\times r_k)/2^k$ is a metric inducing the vague unifomity $B$. We leave the verification to the reader.
	\end{proof} 
\subsection{Vague compactness of Radon measures}\label{subsec_vague_compact}
Let $X$ be a locally compact Polish space. We begin by showing that on a compact space, the space of Radon measures with constant mass is vaguely compact. 
\begin{proposition}\label{prop_comp_measure}
	Assume that $X$ is a compact topological space. Then the set $\{\nu\in \mathcal{M}(X)\;:\;\nu(X)=M\}$ is vaguely compact.
\end{proposition}
\begin{proof}
	Observe that the function which is identically equal to $1$ on $X$, which we shall denote by $\mathbb{1}_X$, belongs to $C_c(X)$. To show that  $\{\nu\in \mathcal{M}(X)\;:\;\nu(X)=M\}$ is vaguely closed, it suffices to show that it is vaguely sequentially closed since it is metrizable. Let $(\nu_ n)_{n\in \N}$ be a sequence in $\{\nu\in \mathcal{M}(X)\;:\;\nu(X)=M\}$ and assume that $\nu_n$ converges vaguely to $\nu'$ as $n\to+\infty.$ Then we have 
	\begin{equation}
M=	\lim_{n\to+\infty}\nu_n(X)=\lim_{n\to+\infty}\int_X \mathbb{1}_X(x)\nu_n(dx)=\lim_{n\to+\infty}\int_X \mathbb{1}_X(x)\nu'(dx)=\nu'(X).
	\end{equation}
	This shows that $\{\nu\in \mathcal{M}(X)\;:\;\nu(X)=M\}$ is vaguely closed. It suffices to show that $\{\nu\in \mathcal{M}(X)\;:\;\nu(X)=M\}$ is vaguely precompact, and since $\mathcal{M}(X)$ is metrizable, it is enough to show sequential vague precompactness. By Theorem \ref{thm_separable}, there exists a countable and dense subset $\mathscr{D}$ of $C_c(X)$.  Let $(\nu_n)_{n\in\N}$ be a sequence in $\{\nu\in \mathcal{M}(X)\;:\;\nu(X)=M\}$. For every $\Phi\in \mathscr{D}$, we have 
	\begin{equation}
	\Big|\int_X \Phi(x)\nu_n(dx)\Big|\leq M\|\Phi\|_{C^0}. 
	\end{equation}
	By a diagonal argument, there exists an increasing map
$\d:\N\to\N$ and a countinuous linear map $\Lambda:\mathscr{D}\to \R$ such that 
	\begin{equation} 
	\lim_{n\to+\infty}\int_X\Phi(x)\nu_{\d(n)}(dx)=\Lambda(\Phi) \qquad\forall \Phi\in \mathscr{D}.
	\end{equation} 
	$\Lambda$ admits a unique continuous extension $\bar\Lambda$ to $C_c(X)$, for which $\Phi\geq 0$ implies $\bar\Lambda(\Phi)\geq 0$. Therefore by the Riesz-Markov-Kakutani theorem, there exists a unique Radon measure $\nu'\in \mathcal{M}(X)$ such that 
	\begin{equation}
	\lim_{n\to+\infty}\int_X\Phi(x)\nu_{\d(n)}(dx)=\bar\Lambda(\Phi) =\int_X\Phi(x)\nu'(dx)\qquad\forall\Phi\in C_c(X). 
	\end{equation}
\end{proof}
 We then have the following corollary.
\begin{corollary}\label{cor_sigma_compact}
Let $(A_k)_{k\in\N}$ be a countable exhaustion by compacts of $X$, and let $(M_k)_{k\in\N}$ be sequence of nonnegative real numbers. Then the set $$B:=\bigcap_{k\in\N}\Big\{\nu\in \mathcal{M}(X)\;:\;\nu(A_k)=M_k\Big\}$$
	is vaguely compact. 
\end{corollary}
\begin{proof}
	The case when $B=\emptyset$ is trivial. 
	Observe that for every $A_k$, the restriction map $\mathcal{M}(X)\ni \nu \longmapsto \nu\lfloor_{A_k}\in  \mathcal{M}(A_k)$ is vaguely continuous.
Therefore, in view of Proposition \ref{prop_comp_measure}, the set $\{\nu\in \mathcal{M}(X)\;:\;\nu(A_k)=M_{k}\}$ is vaguely closed, whence 
$B$ is also vaguely closed. As $B$ is metrizable by Lemma \ref{lem_vague_unif_metrizable_compact_exhaus}, it is now enough to show that $B$ is sequentially vaguely precompact in $\mathcal{M}(X)$. Let $(\nu_n)_{n\in\N}$ be a sequence in $B$. In view of Proposition \ref{prop_comp_measure}, $(\nu_n\lfloor _{A_k})_{n\in\N}$ is vaguely sequentially precompact in $\mathcal{M}(A_k)$ for every $k\in\N$. By a diagonal argument, there thus exists an increasing map $\d:\N\to\N$ such that for every $k\in\N$, we have $\nu_{\d(n)}\lfloor_{A_k} $ vaguely converges to $\nu\lfloor_{A_k}$ as $n\to+\infty$ in $\mathcal{M}(A_k)$. Observe that $\nu\lfloor_{A_{k_1}}\lfloor_{A_{k_0}}=\nu\lfloor_{A_{k_0}}$ if $k_0\leq  k_1$. Denoting by $\iota_{A_k} :A_k\hookrightarrow X$ the canonical injections, we can then set $\nu=\lim_{k\to+\infty}(\iota_{A_k})_\# \nu_{A_k}$ in $\mathcal{M}(X)$. Let us finally show that $\nu_{\d(n)}$ vaguely converges to $\nu$ in $\mathcal{M}(X)$ as $n\to+\infty$. Let $\Phi\in C_c(X)$ and let $k\in\N$ such that $\supp\Phi\subset A_k$, whence we have
\begin{equation}
\begin{split} 
\lim_{n\to+\infty}\int_X\Phi(x)\nu_{\d(n)}(dx)&=\lim_{n\to+\infty}\int_{A_k}\Phi(\iota_{A_k}x)\nu_{\d(n)}\lfloor_{A_k}(dx)\\
&=\int_{A_k}\Phi(\iota_{A_k} x)\nu\lfloor_{A_k}(dx)\\
&=\int_X\Phi(x)\nu(dx). 
\end{split}
\end{equation}
The thesis follows. 
	\end{proof} 

\subsection{Disintegration of Radon measures}\label{subsec_disint}
Let us now present a useful tool for our purpose: the disintegration of a measure with respect to a Borel map and a target measure used in the study of linear transport in \cite{ABC14,BianchiniBonicatto20,Pitcho_Int,Pitcho_SDE}. A presentation with proofs of the disintegration of a measure is given in \cite{DellacherieMeyer}. 
Let $X$ and $Y$ be a locally compact Polish spaces, $\mu$ a Radon measure on $X$, $\nu$ a Radon measure on $Y$ and $f:X\to Y$ a Borel map such that $f_\#\mu=\nu$. Then there exists a Borel family of probability measures $\{\mu_y:y\in Y\}$ of measures on $Y$ such that 
\begin{enumerate}
	\item $\mu_y$ is concentrated on the level set $E_y:=f^{-1}(y)$ for every $y\in Y$;
	\item the measure $\mu$ can be decomposed as $\mu=\int_Y\mu_y\nu(dy)$ , which means that for every Borel set $A$ contained in $X$, we have
	\begin{equation}
	\mu(A)=\int_Y\mu_y(A)\nu(dy). 
	\end{equation}
\end{enumerate}
Any family satisfying $(i)$ and $(ii)$ is called a \emph{disintegration} of $\mu$ with respect to $f$ and $\nu$. The disintegration is essentially unique in the following sense: for any other disintegration $\{\tilde\mu_y: y\in Y\}$, it holds $\mu_y=\tilde\mu_y$ for $\nu$-a.e. $y\in Y$. 
We also have
\begin{equation}\label{eqn_int_disintegration}
\int_X \phi d\mu=\int_Y\Big[\int_{E_y}\phi d\mu_y\Big]\nu(dy)
\end{equation}
for every $\phi \in L^1(X,\mu)$. 

\bigskip 

We now give a useful fact.
Let $g:X\to X$ and $h:X\to Y$ be Borel maps 
such that:
\begin{itemize}
	\item[(P)] $h_\#\mu=\nu$ and for every $y\in Y$, we have $g^{-1}((h^{-1}(y))^c)=(f^{-1}(y))^c$. 
\end{itemize} 
The following is true. 
\begin{lemma} \label{lem_disintegration_push_forward} 
If $\{\mu_y : y\in Y\}$ is a disintegration of $\mu$ with respect to $ f$ and $\nu$, then $\{g_\# \mu_y : y\in Y\}$ is a disintegration of $g_\#\mu$ with respect to $ h$ and $\nu$. 
\end{lemma} 
\begin{proof}
	Let $y\in Y$. Observe that $g_\#\mu_y((h^{-1}(y))^c)=\mu_y(g^{-1}(h^{-1}(y))^c)=\mu_y((f^{-1}(y))^c)=0$ where we have used (P) in the second to last equality and that $\mu_y$ is concentrated on $f^{-1}(y)$ in the last equality. 
	So	$g_\#\mu_y$ is supported on $h^{-1}(y)$, and since $y$ was arbitrary, this proves $(i)$. 
	
	Let $A$ a Borel set in $X$. Then as $g^{-1}(A)$ is a Borel set in $X$, it follows that 
	\begin{equation}
	g_\#\mu(A)=\mu(g^{-1}(A))=\int_Y \mu_y (g^{-1}(A))\nu(dy)=\int_Y g_\#\mu_y (A)\nu(dy),
	\end{equation}
	which gives $(ii)$. 
\end{proof}
Let us denote by the support of $\nu$ by $\supp\nu$, which is the smallest closed set such that $\nu(Y\backslash \supp\nu)=0$.
By a slight abuse of language, the restriction $\{\mu_y\;:\;y\in \supp\nu\}$ of $\{\mu_y\;:\;y\in Y\}$ to $\supp\nu$ will also be called a disintegration of $\mu$ with respect to $f$ and $\nu$. 
\subsection{Measures on the space of continuous paths }\label{subsec_measure_space_cont_paths}
Let $Y$ be a locally compact Polish space. Consider the uniformity of compact convergence $\mathcal{C}_c(\R;Y)$ on the set of continuous maps from $\R$ to $Y$, which is Hausdorff  and complete by Theorem \ref{theorem_hausdorff_complete}. Precompact sets with respect to the uniformity of compact convergence are characterised by Ascoli's theorem. 

\subsubsection{The space of paths $\G$}\label{subsec_space_G} 
Given $h\in \R$, we define the map $b_h(\g(\cdot))=\g(\cdot-h)$ from the set of continuous maps from $\R$ to $Y$, into the same set. 
Let $\Gamma$ be a subset of $\mathcal{C}_c(\R;Y)$ such that:
\begin{enumerate} 

\item[(Comp-$\G$)] for each compact $K\subset Y$, and each compact $I\subset \R$, $\cup_{t\in I}e_t^{-1}(K)\cap \G$ is compact in $\mathcal{C}_c(\R;Y);$
\item[(t-trans-$\G$)]
for every $h\in \R$, we have $b_h(\G)=\G$.
\end{enumerate}
 We endow $\G$ with the induced uniformity. 
 \begin{itemize}
 	\item  $\Gamma$ is locally compact. Indeed, let $\gamma\in \G$. Since $Y$ is locally compact, there exists an open set $U$ in $ Y$ such that $\gamma(0)\in U$, and a compact set $K$ in $Y$ contaning $U$, whence we have
 	$$\gamma\in e_0^{-1}(U)\cap \G\subset e_0^{-1}(K)\cap \G,$$ and by continuity of $e_0$, the set $e_0^{-1}(U)\cap \G $ is open in $\G$, and $e_0^{-1}(K)\cap \G$ is compact in $\G$ by (Comp-$\G$).
 	
 	\item $\G$ is metrizable. This follows directly from Theorem \ref{thm_separable}. 
 	
 	\item $\G$ is completely metrizable. 
 	Let ${\rm d}$ be a metric compatible with the uniform structure on $\G$. Let $(\g_n)_{n\in\N}$ be a Cauchy sequence in $(\G,{\rm d})$. Then $(e_0(\g_n))_{n\in\N}$ is Cauchy in $Y$ since the map $e_0$ is uniformly continuous, and converges to some $y\in Y$ by completeness of $Y$. Since $Y$ is locally compact, there is an open set $U$, and a compact set $K$ such that $y\in U\subset K,$ and up to changing the starting index of the sequence, we can assume that $(e_0(\g_n))_{n\in\N}\subset U$. Therefore $(\g_n)_{n\in\N}\subset e_0^{-1}( K)$, so by compactness of $e_0^{-1}(K)$, the sequence $(\g_n)_{n\in\N}$ converges in $(\G,{\rm d})$ to some $\g\in e_0^{-1}(K)$. 
 	\item  $\G$ is separable, since $\mathcal{C}_c(\R;Y)$ is separable by Theorem \ref{thm_separable}.
 \end{itemize}
 Therefore $\G$ is a locally compact Polish space. 
 
\subsubsection{$\mu$-invariant Radon measures on $\G$}

Let $\mu$ be a Radon measure on $Y$ and let $\mathcal{M}_\mu(\G)$ be the subset of $\mathcal{M}(\G)$ such that for every $\eeta\in \mathcal{M}(\G)$, and every $t\in \R$, we have $(e_t)_\#\eeta=\mu$. $\mathcal{M}_\mu(\G)$ is a uniform space endowed with the initial uniformity with respect to the canonical injection $\mathcal{M}_\mu(\G)\hookrightarrow \mathcal{M}(\G)$. 
Observe that a fundamental system of entourages of the uniformity on $\mathcal{M}_\mu(\G)$ is given by 
\begin{equation}\label{eqn_fundamental_syst_M_mu}
\bigcap_{i=1}^N{\rm d}^{-1}_{\Phi_i}([0,\a])\cap \mathcal{M}_\mu(\G)\times \mathcal{M}_\mu(\G),
\end{equation}
where $\a$ runs in $\Q^+$, $N$ runs in $\N$, and $\Phi_i$ run in $C_c(\G)$. 
We record the following properties of $\mathcal{M}_\mu(\G)$. 
\begin{lemma}\label{lem_mu_invariant_space}
The vague uniformity	on $\mathcal{M}_\mu(\G)$ is metrizable and compact.
\end{lemma}
\begin{proof}
	Let $(B_k)_{k\in\N}$ be an exhaustion by compacts of $Y$. 
	Setting $A_k=e_0^{-1}(B_k)$, which is compact in view of (Comp-$\G$), and setting $M_k=\mu(e_0^{-1}(B_k))$ for every $k\in\N$, by Lemma \ref{lem_vague_unif_metrizable_compact_exhaus}, the vague uniformity on 
	\begin{equation}
	B:=\bigcap_{k\in\N}\{\eeta\in \mathcal{M}(\G)\;:\;\eeta(A_k)=M_k\} 
	\end{equation}
	is metrizable, and since
	$\mathcal{M}_\mu(\G)\subset B$, the vague uniformity on $\mathcal{M}_\mu(\G)$ is also metrizable. 
	
In view of Corollary \ref{cor_sigma_compact}, the set $B$
	is vaguely compact. To conclude the proof let us show that $\mathcal{M}_\mu(\G)$ is vaguely closed. Note that it suffices to show that  $\mathcal{M}_\mu(\G)$ is vaguely sequentially closed by metrizability. Let $(\eeta_n)_{n\in\N}$ be a sequence in $\mathcal{M}_\mu(\G)$ such that $\eeta_n$ converges vaguely to $\eeta\in \mathcal{M}(\G)$ as $n\to+\infty$. Let $\phi\in C_c(Y)$ and $t\in \R$, and observe that $\phi\circ e_t\in C_c(\G)$, whence 
	\begin{equation}
	\int_\G \phi(e_t(\g))\eeta(d\g)=	\lim_{n\to+\infty}\int_\G \phi(e_t(\g))\eeta_n(d\g)=\int_\G\phi(x)\mu(dx).
	\end{equation}
	Since $\phi$  was arbitrary, it follows that $(e_t)_\#\eeta=\mu$, and the thesis follows. 
	
	\bigskip 

\end{proof}

\subsubsection{Radon probability measures}\label{subsec_radon_proba}
Let us denote by $\mathcal{P}(\G)$ the subset of $\mathcal{M}(\G)$ consisting of probability measures, i.e. $\mu\in \mathcal{P}(\G)$ if and only if $\mu(\G)=1.$ We endow $\mathcal{P}(\G)$ with the initial uniform structure with respect to the canonical injection $\mathcal{P}(\G)\hookrightarrow \mathcal{M}(\G).$ By Lemma \ref{lem_separability}, the vague topology on $\mathcal{P}(\G)$ is metrizable and separable. Let $\mathscr{D}$ be a countable and dense subset of $ C_c(\G)$, which by Theorem \ref{thm_separable} exists. 
A fundamental system of entourage $\mathfrak{V}$ of $\mathcal{P}(\G)$ is then given by 
\begin{equation}
\bigcap_{i=1}^N{\rm d}_{\Phi_i}^{-1}([0,\a]) \cap \mathcal{P}(\G)\times \mathcal{P}(\G),
\end{equation}
where $\Phi_i$ runs in $\mathscr{D}$, $N$ runs in $\N$ and $\a$ runs in $\Q^+$.

\bigskip 
\section{The tensor product of measures} \label{sec_tensor_prod_measure}
In the section, we will construct the tensor product of measures through a map. 
\subsection{Generalities on measure theory}
A \emph{measurable space} is a pair consisting of a set and a $\s$-algebra over this set. 
A map between two measurable spaces is said to be \emph{measurable}, if under this map, the inverse image of every set in the $\s$-algebra of the codomain lies in the $\s$-algebra of the domain. A \emph{measure} on a measurable space is a real-valued, nonnegative set function on the $\s$-algebra such that the image of the empty set is zero and which is countably additive. A measurable space endowed with a measure is called a \emph{measure space}.  

Let $(X,\Sigma)$ and $(Y,\Sigma')$ be two measurable spaces, let $\nu$ be measure on $(X,\Sigma)$ and let $g$ be a measurable map from $X$ to $Y$. The \emph{pushforward measure} of $\nu$ through $g$ is defined to be the unique measure $g_\#\nu$ on $(Y,\Sigma)$ such that 
\begin{equation}
g_\#\nu(A)=\nu(g^{-1}(A))\qquad\forall A\in \Sigma'.
\end{equation}
Let $\nu'$ be a measure on $(Y,\Sigma')$. 
Assume now that $g$ also satisfies the following property $g(A)\in \Sigma'$ for every $A\in \Sigma$. 
The \emph{pullback measure }of $\nu'$ through $g$ is defined to be the unique measure $g^\#\nu'$ on $(X,\Sigma)$ such that 
\begin{equation}
g^\#\nu'(A)=\nu'(g(A))\qquad\forall A\in \Sigma. 
\end{equation}
It then holds that $g^\#g_\#\nu=\nu$ if $g$ is injective, and $g_\#g^\#\nu'=\nu'$ if $g$ is surjective.

\subsection{Tensor product of measures through a map}\label{subsec_tensor_product_through_map}
Let $(X_1,\Sigma_1)$ and $(X_2,\Sigma_2)$ be two measurable spaces.
Consider the product measurable space $(X_1\times X_2,\Sigma_1\otimes \Sigma_2)$, where the product $\s$-algebra $\Sigma_1\otimes \Sigma_2$ is the coarsest $\s$-algebra, which makes the two canonical projections maps $p_1$ and $p_2$ are measurable. 
Let $X$ be a set, let $f:X\to X_1\times X_2 $ be an injective map, and let $\Sigma$ be the coarsest $\s$-algebra on $X$, which makes $f$ measurable. 
Let $\nu_1$ and $\nu_2$ be two measures on $(X_1,\Sigma_1)$ and $(X_2,\Sigma_2)$ respectively. 

\begin{proposition} \label{prop_tensor}
Assume that there exists $S_1\in \Sigma_1$ and $S_2\in \Sigma_2$ such that 
\begin{enumerate} 
\item$X_1\times X_2\backslash S_1\times S_2\subset  {\rm Im} f;$  
\item $\nu_1(S_1)=0$ and $\nu_2(S_2)=0$. 
\end{enumerate} 
Then there exists a unique probability measure $\nu$ on $(X,\Sigma)$ such that $f_\#\nu=\nu_1\times\nu_2$, which explicitely means that for every set $A_1\in \Sigma_1$ and $A_2\in \Sigma_2$, we have
\begin{equation}\label{eqn_prod}
\nu\big((p_1\circ f)^{-1}(A_1)\cap (p_2\circ f)^{-1}(A_2)\big)=\nu_1(A_1)\nu_2(A_2).
\end{equation}
$\nu$ is called the tensor product and $\nu_1$ and $\nu_2$ through $f$ and denoted by $\nu_1\otimes_f\nu_2$. 
\end{proposition}
\begin{proof}
Let $\bar X:=f^{-1}(X_1\times X_2\backslash S_1\times S_2)\in \Sigma$. Let $\bar\iota:\bar X\hookrightarrow X$ and $\iota: X_1\times X_2\backslash S_1\times S_2\hookrightarrow X_1\times X_2$ be the canonical injections, and let $\bar\Sigma$ be the initial $\s$-algebra on $X_1\times X_2\backslash S_1\times S_2$  with respect to $\iota$. 
	Note that there exists a unique measurable isomorphism $\bar f:\bar X\to X_1\times X_2\backslash S_1\times S_2$ such that $$  \iota\circ \bar f =f\circ \bar \iota .$$
Note that $\iota(A)$ lies in $\Sigma_1\otimes \Sigma_2$ for each set $A$ in $\bar\Sigma$ since the projection $\iota(A)$ onto $X_1$  (resp. $X_2$) belongs to $\Sigma_1$ (resp. $\Sigma_2$). Also $\iota^\#\nu_1\times \nu_2=\nu_1\times\nu_2\lfloor_{X_1\times X_2\backslash S_1\times S_2}.$ 
	Therefore by $(ii)$, we have
	\begin{equation}\label{eqn_iota}
	\iota_\#\iota^\#(\nu_1\times \nu_2)=\nu_1\times \nu_2. 
	\end{equation} 
	Define the measure $ \nu=\bar\iota_\#\bar f^\#\iota^\#\nu_1\times \nu_2$ on $ X$. Since $\bar f$ is a measurable isomorphism and by \eqref{eqn_iota}, we have
	\begin{equation}\label{eqn_charact_mu}
	f_\#\nu=f_\#\bar\iota_\#\bar f^\#\iota^\#\nu_1\times \nu_2=\iota_\# \bar f_\#\bar f^\#\iota^\#\nu_1\times \nu_2 =\iota_\# \iota^\#\nu_1\times \nu_2=\nu_1\times \nu_2.  
	\end{equation}
	Let us prove that \eqref{eqn_charact_mu} uniquely characterises $\nu$. Let $\bar\nu$ satisfy $f_\#\bar\nu=\nu_1\times \nu_2$, and observe that, we have $\bar f^\#\iota^\#f_\#\bar\nu=\bar\iota^\#\bar\nu$, whenever $\bar\nu$ satisfies $(ii)$. 
	Therefore,
	\begin{equation}
\bar\nu=\bar\iota_\#\bar\iota^\#\bar\nu=\bar\iota_\#\bar f^\#\iota^\#f_\#	\bar\nu=\bar\iota_\#\bar f^\#\iota^\#\nu_1\times \nu_2,
	\end{equation}
	which shows that $\nu$ is uniquely characterised.

	\bigskip 
\end{proof}


\bigskip 
\subsection{Tensor product and composition of maps}
Let $\bar X_1$ and $\bar X_2$ be two sets. Let $g_1:X_1\to \bar X_1$ and $g_2:X_2\to \bar X_2$ be two surjective maps and let $\bar \Sigma_{1}$ and $\bar \Sigma_{2}$ be $\s$-algebras on $\bar X_1$ and $\bar X_2$, which respectively make $g_1$ and $g_2$ measurable. Under the hypothesis of Proposition \ref{prop_tensor}, assume that we have
\begin{equation} \label{eqn_assump}
 g^{-1}_1(g_1(S_1))=S_1 \qquad\text{ and }\qquad g^{-1}_2(g_2(S_2))=S_2. 
 \end{equation} 
Let $\bar X$ be a set and let $g:X\to \bar X$ be a surjective map. Let $\bar \Sigma$ be a $\s$-algebra on $\bar X$, which makes $g$ measurable. Assume further that there exists an injective and measurable map $\bar f :\bar X\to \bar X_1\times \bar X_2$ such that for every $A\in \bar\Sigma$, we have $\bar f(A)\in \bar\Sigma_1\otimes \bar \Sigma_2$, and further assume that
\begin{equation}\label{eqn_assump_2}
\bar f\circ g=(g_1\times g_2)\circ f.
\end{equation}
\begin{lemma}\label{lem_tensor_proj}
	Under the same hypothesis as Proposition \ref{prop_tensor}, the tensor product of $(g_1)_\#\nu_1$ and $(g_2)_\#\nu_2$ through $\bar f$ is well-defined, denoted by $(g_1)_\#\nu_1\otimes_{\bar f} (g_2)_\#\nu_2$, and characterised by 
	\begin{equation}
\bar f_\#[	(g_1)_\#\nu_1\otimes_{\bar f} (g_2)_\#\nu_2]=(g_1)_\#\nu_1\times (g_2)_\#\nu_2,
	\end{equation}
	and further satisfies
	\begin{equation}\label{eqn_satisfies}
	(g_1)_\#\nu_1\otimes_{\bar f} (g_2)_\#\nu_2=g_\#\nu_1\otimes_f\nu_2. 
	\end{equation}
\end{lemma} 
\begin{proof}
	Set $\bar S_1:=g_1(S_1)$ and $\bar S_2:=g_2(S_2)$ and observe that $\bar S_1\in \bar\Sigma_1$ and $\bar S_2\in \bar\Sigma_2$ by \eqref{eqn_assump}, and that $$(g_1)_\#\nu_1(\bar S_1)=0=(g_2)_\#\nu_2(\bar S_2).$$
Since $g$ is surjective, we have that ${\rm Im}(\bar f \circ g)={\rm Im }\bar f$. 
	Therefore $\bar X_1\times \bar X_2\backslash \bar S_1\times \bar S_2\subset {\rm Im} \;\bar f$ since $g_1\times g_2$ is surjective, since $X_1\times X_2\backslash S_1\times S_2\subset {\rm Im } f$, and in view of \eqref{eqn_assump} and \eqref{eqn_assump_2}. Therefore the tensor product of $(g_1)_\#\nu_1$ and $(g_2)_\#\nu_2$ through $\bar f$ is well-defined by Proposition \ref{prop_tensor}, and is characterised by
	\begin{equation}
	\bar f_\#[ (g_1)_\#\nu_1\otimes_{\bar  f} (g_2)_\#\nu_2]=(g_1)_\#\nu_1\times (g_2)_\#\nu_2=(g_1\times g_2)_\#\nu_1\times \nu_2. 
	\end{equation}
	Also, since $\bar f$ is injective, we have that 
	\begin{equation}
	(g_1)_\#\nu_1\otimes_{\bar  f} (g_2)_\#\nu_2=\bar f^\# (g_1\times g_2)_\#\nu_1\times \nu_2. 
	\end{equation}
	Since $f$ is injective, by the characterising property of tensor product of Proposition \ref{prop_tensor}, we have that 
	\begin{equation}
g_\#	\nu_1\otimes_f \nu_2=g_\#f^\#\nu_1\times \nu_2. 
	\end{equation}
	Notice now that 
	\begin{equation}\label{eqn_Above}
	g_\#f^\#\nu_1\times \nu_2=\bar f^\# (g_1\times g_2)_\#\nu_1\times \nu_2,
	\end{equation}
holds if and only 
	\begin{equation}\label{eqn_equivalence}
(g_1\times g_2)_\#f_\#f^\#\nu_1\times \nu_2= (g_1\times g_2)_\#\nu_1\times \nu_2,
\end{equation}
where we have used \eqref{eqn_assump_2} and $\bar X_1\times \bar X_2\backslash \bar S_1\times \bar S_2\subset {\rm Im}\bar f$. Since $X_1\times X_2\backslash S_1\times S_2\subset {\rm Im} f$, it holds that $f_\#f^\#\nu_1\times\nu_2=\nu_1\times\nu_2$, whence \eqref{eqn_equivalence} holds and thus in view of \eqref{eqn_Above}, we have that 
\begin{equation}
(g_1)_\#\nu_1\otimes_{\bar f}(g_2)_\#\nu_2=g_\#\nu_1\otimes_f \nu_2. 
\end{equation}
This proves the lemma. 
\end{proof}

\bigskip 
\section{Markovianisation} \label{sec_markovianisation}
Let $Y$ be a locally compact Polish space. Let $\mu$ be a Radon measure on $Y$. Let $\Gamma$ be as in Section \ref{subsec_space_G}. 
In this section, we present Markovianisation of measures in $\mathcal{M}_\mu(\G)$ at a single time, and at a finite number of times.
\subsection{The Markov operator at a single time}\label{subsect_def_markov}
Let $I\subset \R$ be an interval, let  $\iota^I$ be the canonical injection $I \hookrightarrow \R$. Define the map $r^I$ from $C(\R;Y)$ to  $C(I;Y)$ by $r^I(\g)=\g\circ \iota^I$ for every $\g\in \G$. Let $\G\lfloor_I$ be the image of $\G$ under $r^I$, and by a slight abuse of notation, let $ r^I$ be the surjective restriction of $r^I$ to $\G$ into $\G\lfloor_I$. $\G\lfloor_I$ is then endowed with the final uniformity with respect to $r^I$. The pushforward of probability measures through $r^I$  induces a surjective map $(r^I)_\#$ from $\mathcal{P}(\G)$ to $\mathcal{P}(\G\lfloor_I)$. Given $\eeta\in \mathcal{P}(\G)$, we will write $\eeta^I:=(r^I)_\#\eeta$ for the pushforward of $\eta$ through through the map $r^I$.

Let $I_1,I_2$ be two closed intervals in $\R$ such that $I= I_1\cup I_2$ and $I_1\cap I_2=\{t\}$ for some $t\in \R$, and $t_1\leq t_2$ for every $t_1\in I_1$ and every $t_2\in I_2$. Let $\eeta_1\in \mathcal{P}(\G\lfloor_{I_1})$ and let $\eeta_2\in \mathcal{P}(\G\lfloor_{I_2})$. Assume that there exists $x\in Y$ such that $$\eeta_1(\{\g\in \G\lfloor_{I_1}\;:\;\g(t)\neq x\})=0=\eeta_2(\{\g\in\G\lfloor_{I_2}\;:\;\g(t)\neq x\}).$$
The hypothesis of Proposition \ref{prop_tensor} are then satisfied with $(X_1,\Sigma_1)= (\mathcal{P}(\G\lfloor_{I_1}),\mathscr{B}(\G\lfloor_{I_1}))$,   $(X_2,\Sigma_2)= (\mathcal{P}(\G\lfloor_{I_2}),\mathscr{B}(\G\lfloor_{I_2}))$, $f=r^{I_1}\times r^{I_2}$, $S_1=\{\g\in\G\lfloor_{I_1}\;:\;\g(t)\neq x\}$, and $S_2=\{\g\in \G\lfloor_{I_2}\;:\;\g(t)\neq x\}$. Therefore, as a probability measure in $\mathcal{P}(\G\lfloor_I)$, the tensor product of $\eeta_1$ and $\eeta_2$ through $r^{I_1}\times r^{I_2}$ is well-defined and we shall denote it by $\eeta_1\otimes_{(t,x)}\eeta_2$.


 Let $\eeta\in \mathcal{M}_\mu(\G)$ and for every $t\in \R$, let $\{\eta_{t,x}\}_{x\in Y}$ be a disintegration of $\eeta$ with respect to $e_t$ and $\mu$. Let $t\in \R$, and define the Markovianisation map $M_t$ at time $t$ from $\mathcal{M}_\mu(\G)$ to $\mathcal{M}_\mu(\G)$ as follows:
 \begin{equation}
 M_t(\eeta)=\int_{Y} \eta^{(-\infty,t]}_{t,x} \otimes_{(t,x)} \eta^{[t,+\infty)}_{t,x}\mu(dx). 
 \end{equation}
 Observe that the definition of $M_t$ depends on the measure $\mu\in\mathcal{M}(Y)$. 
\begin{lemma}
	The map $M_t:\mathcal{M}_\mu(\G)\longrightarrow\mathcal{M}_\mu(\G)$ is well-defined. 
\end{lemma}
\begin{proof}
	Let us show that $(e_s)_\#M_t(\eeta)=\mu$ for every $s\in \R$. For $t=s$, this follows from the fact that $\eta^{(-\infty,t]}_{t,x} \otimes_{(t,x)} \eta^{[t,+\infty)}_{t,x}$ is a probability measure concentrated on $\{\g\in\G\;:\;\g(t)=x\}$. Let $s\neq t$ and by a slight abuse of notation let $e_s$ denote both the evaluation map at time $s$ from $\G$ to $Y$ and the evaluation map at time $s$ from $\G\lfloor_{(-\infty,t]}\times \G\lfloor_{[t,+\infty)}$ to $Y$. 
	Observe that $e_s\circ f=e_s$.
	For each $x\in Y$, we thus have by the characterising property of Proposition \ref{prop_tensor}
	\begin{equation}
	(e_s)_\# [\eta^{(-\infty,t]}_{t,x}\otimes_{(t,x)}\eta^{[t.+\infty)}_{t,x}]=	(e_s\circ f)_\# [\eta^{(-\infty,t]}_{t,x}\otimes_{(t,x)}\eta^{[t.+\infty)}_{t,x}]=(e_s)_\#\eta_{t,x},
	\end{equation}
	whereupon integrating against $\mu(dx)$, we find $(e_s)_\#M_t(\eeta)=(e_s)_\#\eeta=\mu$. This shows that $M_t(\eeta)\in \mathcal{M}_\mu(\G).$
	The fact that $M_t(\eeta)$ does not depend on the choice of the disintegration of $\eeta$ with respect to $e_t$ and $\mu$ follows directly from essential uniqueness of the disintegration. The thesis is proved.
	\bigskip 
	
\end{proof}
\subsection{Bilinearity of the tensor product}
Let $I_1$ be an interval in $\R$ of the form $[t_0,t_1]$ or $(-\infty,t_1]$ for some real numbers $t_0<t_1$, and let $I_2$ be an interval in $\R$ of the form $[t_1,t_2]$ or $[t_1,+\infty)$ for some real number $t_2$. Let also $I:=I_1\cup I_2$, and let $\eeta_1\in \mathcal{P}(\G\lfloor_{I_1})$, and $\eeta_2\in \mathcal{P}(\G\lfloor_{I_2})$. 
Let $\aalpha: Y\to \mathcal{P}(\G\lfloor_{I_1})$ be a Borel map such that:
\begin{itemize} 
 \item[(supp-disj)] there exists $s\in I_1$ such that for every $y\in Y$, the probability measure $\aalpha(y)$ is supported on $\{\g\in \G\lfloor_{I_1}\;:\;\g(s)=y\}$;
  \item[(supp-$t_1$)] there exists $x\in Y$ such that for every $y\in Y$, the probability measure $\aalpha(y)$ is supported in $\{\g\in \G\lfloor_{I_1}\;:\;\g(t_1)=x\}$.
 \end{itemize} 
\begin{lemma}[Bilinearity]\label{lem_bilinearity}
The following identity holds  
\begin{equation}
\int_{Y}[\aalpha(y)\otimes_{(t_1,x)}\eeta_2]\mu(dy)= \Big[\int_{Y}\aalpha(y)\mu(dy)\Big]\otimes_{(t_1,x)}\eeta_2. 
\end{equation}
\end{lemma}
\begin{remark}
	Swapping the role of $I_1$ and $I_2$, the following identity holds
	\begin{equation}
	\int_{Y}[\eeta_1\otimes_{(t_1,x)}\aalpha(y)]\mu(dy)= \eeta_1\otimes_{(t_1,x)}\Big[\int_{Y}\aalpha(y)\mu(dy)\Big]. 
	\end{equation}
\end{remark}
\begin{proof}
	1. Let us check that both sides of the identity are well-defined. The integral $\int_Y\aalpha(y)\mu(dy)$ is well-defined in $\mathcal{P}(\G\lfloor_{I_1})$, and supported on $\{\g\in \G\lfloor_{I_1}\;:\;\g(t_1)=x\}$ by (supp-$t_1$), whence the tensor product $\Big[\int_{Y}\aalpha(y)\mu(dy)\Big]\otimes_{(t_1,x)}\eeta_2$ is well-defined since the hypothesis of Section \ref{subsect_def_markov} are satisfied. Also the integral $\int_Y[\aalpha(y)\otimes_{(t_1,x)} \eeta_2]\mu(dy)$  is well-defined since for $\mu$-a.e. $y\in Y$, the tensor product $\aalpha(y)\otimes_{(t_1,x)}\eeta_2$ is well-defined because the hypothesis of Section \ref{subsect_def_markov} are satisfied.
	
	\bigskip 
	2. Let us show that 
	\begin{equation}
	\begin{split} 
	f_\#\int_Y [\aalpha(y)\otimes_{(t_1,x)} \eeta_2]\mu(dy)=\int_Y [\aalpha(y)\times \eeta_2]\mu (dy)
		\end{split} 
	\end{equation}
	uniquely characterises $\aalpha(y)\otimes_{(t_1,x)}\eeta_2$ for $\mu$-a.e. $y\in Y$ in the class of Borel maps $\tilde\aalpha$ from $Y$ to $\mathcal{P}(\G\lfloor_I)$ such that $\tilde\aalpha(y)$ is supported on $\{\g\in \G\lfloor_{I}\;:\;\g(s)=y\}$ for every $y\in Y$. 
	Let $\tilde \aalpha :Y\to \mathcal{P}(\G\lfloor_I)$ be a measurable map such that 
\begin{equation}
\begin{split} 
&f_\#\int_Y\tilde \aalpha(y)\mu(dy)=\int_Y[\aalpha(y)\times \eeta_2]\mu(dy) \;\\
&\supp\tilde\aalpha (y)\subset \{\g\in \G\lfloor_I\;:\;\g(s)=y\}\quad\forall y\in Y.  
\end{split} 
\end{equation} 
	Let $\mathscr{G}$ be a countable family generating the $\s$-algebra $\mathscr{B}(\G\lfloor_I)$ and let $B\in \mathscr{G}$.	Let $A\subset Y$ be a Borel set, and let us evaluate the above equation on the Borel set ($r^{I_1})^{-1}(\{\g\in \G\lfloor_{I_1}\;:\;\g(s)\in A\})\cap B$, which by (supp-disj) gives the following
	\begin{equation}
	f_\#\int_{A}\tilde \aalpha(y)(B)\mu(dy)=\int_{A} [\aalpha(y)\times \eeta_2](B)\mu(dy). 
	\end{equation}
	As $A$ was arbitrary, by Lusin's theorem, there exists a $\mu$-negligeable set $N_B\subset Y$ such that $f_\#\tilde\aalpha(y)(B)=[\aalpha (y)\times \eeta_2](B)$ for every $y\in Y\backslash N_B$. We then set $N:=\bigcup_{B\in \mathscr{G}} N_B$, which is also $\mu$-negligeable, and since $\mathscr{G}$ generates the Borel $\s$-algebra $\mathscr{B}(\G\lfloor_I)$, we have that for $\mu$-a.e. $y\in Y$, it holds $f_\#\tilde \aalpha(y)=\aalpha(y)\times \eeta_2$. Therefore, by the characterisation of Proposition \ref{prop_tensor}, we have $\tilde\aalpha (y)=\aalpha (y)\otimes_{(t_1,x)} \eeta_2$ for $\mu$-a.e. $y\in Y$.  
	\bigskip 
	
	3. Let us conclude the proof. 
	By the characterisation of the tensor product of Proposition \ref{prop_tensor}, we have 
	\begin{equation}
	f_\#\Big(\Big[ \int_Y\aalpha(y)\mu(dy)\Big]\otimes_{(t_1,x)} \eeta_2\Big)=\Big[\int_Y \aalpha(y)\mu(dy)\Big]\times \eeta_2=\int_Y [\aalpha(y)\times \eeta_2]\mu(dy),
	\end{equation}
	In view of 2. of this proof the thesis follows. 
\end{proof}

\subsection{Associativity of the tensor product}
Let us now argue that the tensor product is associative. 
\begin{lemma}[Associativity]\label{lem_associativity}
	Let $t_1<t_2$ be real numbers and let $x_1,x_2\in Y$. Let also $I_1,I_2,I_3\subset \R$ be closed intervals such that: $I_1\cap I_2=\{t_1\}$, and $I_2\cap I_3=\{t_2\}$; for every $t_1\in _1$ and every $t_2\in I_2$, we have $t_1\leq t_2$; for every $t_2\in I_2$ and $t_3\in I_3$, we have $t_2\leq t_3$. Let $\eeta_i\in \mathcal{M}(\G\lfloor_{I_i})$ for $i=1,2,3$ such that $\eeta_1$ is concentrated on $\{\g\in \G\lfloor_{I_1}\;:\;\g(t_1)=x_1\}$, $\eeta_3$ is concentrated on $\{\g\in \G\lfloor_{I_3}\;:\;\g(t_2)=x_2\}$ and $\eeta_2$ is concentrated on $\{\g\in \G\lfloor_{I_2}\;:\;\g(t_1)=x_1\text{ and }\g(t_2)=x_2\}$. Then we have the identity
	\begin{equation}
	\eeta_1\otimes_{(t_1,x_1)}[\eeta_2\otimes_{(t_2,x_2)}\eeta_3]= 	[\eeta_1\otimes_{(t_1,x_1)}\eeta_2]\otimes_{(t_2,x_2)}\eeta_3.
	\end{equation}
\end{lemma}
In view of the above lemma, we will write $\eeta_1\otimes_{(t_1,x_1)}\eeta_2\otimes_{(t_2,x_2)}\eeta_3$. 
\begin{proof}
	1. Denote $I=I_1\cup I_2\cup I_3.$ By the characterising property of the tensor product in Proposition \ref{prop_tensor}, $\eeta_2\otimes_{(t_2,x_2)}\eeta_3$ is concentrated on $\{\g\in \G\lfloor_{I_2\cup I_3}\;:\;\g(t_1)=x_1\}$ and $\eeta_1\otimes_{(t_1,x_1)}\eeta_2$ is concentrated on $\{\g\in \G\lfloor_{I_1\cup I_2}\;:\;\g(t_2)=x_2\}$. So both sides of the identity are well-defined by Proposition \ref{prop_tensor}. 
	
	\bigskip 
	
	2. Observe that $ r^{ I_2}\times r^{I_3}\circ r^{I_1}\times r^{I_2\cup I_3}=r^{I_1}\times r^{I_2}\circ r^{I_1\cup I_2}\times  r ^{I_3}=r^{I_1}\times r^{I_2}\times r^{I_3}.$ By Proposition \ref{prop_tensor} each side of the identity we seek to prove are uniquely characterised by
	\begin{equation}
	\begin{split} 
(	r^{I_1}\times r^{I_2}\times r^{I_3})_\#	\eeta_1\otimes_{(t_1,x_1)}[\eeta_2\otimes_{(t_2,x_2)}\eeta_3]&=\eeta_1\times \eeta_2\times \eeta_3,\\
(	r^{I_1}\times r^{I_2}\times r^{I_3})_\#	[\eeta_1\otimes_{(t_1,x_1)}\eeta_2]\otimes_{(t_2,x_2)}\eeta_3&=\eeta_1\times \eeta_2\times \eeta_3,
\end{split} 
	\end{equation}
	whence they must be equal, and this proves the thesis. 
\end{proof}
\subsection{The Markov operator at several times}\label{subsec_markov_several_time}
Let us begin with the following  lemma. 
\begin{lemma}[Commutativity]
Let $t_1,t_2\in \R$. Then we have $(M_{t_1}\circ M_{t_2})=(M_{t_2}\circ M_{t_1})$.
\end{lemma}
\begin{proof}
	Let $\eeta\in \mathcal{M}_\mu(\G)$. Without loss of generality, assume that $t_1<t_2$. Define $\nu:=(e_{t_1}\times e_{t_2})_\#\eeta$ in $\mathcal{M}(Y\times Y)$. Denote the projections 
	\begin{equation}
	\begin{split}
	q_1&:Y\times Y\ni (x,y)\longmapsto x\in Y,\\
	q_2&:Y\times Y\ni (x,y)\longmapsto y\in Y.
	\end{split}
	\end{equation}
	We observe that $(q_1)_\#\nu=\mu=(q_2)_\#\nu$. Let $\{\nu_x\}_{x\in Y}$ (resp. $\{\nu_y\}_{y\in Y}$) be a disintegration of $\nu$ with respect to $q_1$ (resp. $q_2$) and $\mu$. 
Let $\{\eta_{t_1,x}\}_{x\in Y}$ (resp. $\{\eta_{t_2,y}\}_{y\in Y}$)	be disintegrations of $\eeta$ with respect to $e_{t_1}$ (resp. $e_{t_2}$) and $\mu$. Let $\{\eta_{t_1,x;t_2,y}\}_{(x,y)\in Y\times Y}$ be a disintegration of $\eta$ with respect to $e_{t_1}\times e_{t_2}$ and $\nu$.  Let us show that
	\begin{equation}
	(	M_{t_2}\circ M_{t_1})(\eeta)=\int_{Y\times Y}\eta^{(-\infty,t_1]}_{t_1,x}\otimes_{(t_1,x)}\eta^{[t_1,t_2]}_{t_1,x,;t_2,y}\otimes_{(t_2,y)}\eta^{[t_2,+\infty)}_{t_2,y}\;\nu(dx,dy).
	\end{equation}
	By essential uniqueness of the disintegration, observe that for $\mu$-a.e. $x\in Y$, we have
	\begin{equation}
	\eta^{[t_1,+\infty)}_{t_1,x}=\int_{Y}\eta^{[t_1,+\infty)}_{t_1,x;t_2,y}\nu_x(dy). 
	\end{equation}
	Therefore, by Lemma \ref{lem_bilinearity}, and Fubini's theorem, we have
	\begin{equation}
	\begin{split} 
	M_{t_1}(\eeta)&=\int_{Y}\eta_{t_1,x}^{(-\infty,t_1]}\otimes_{(t_1,x)}\eta_{t_1,x}^{[t_1,+\infty)}\mu(dx)\\
	&=\int_{Y\times Y}\eta_{t_1,x}^{(-\infty,t_1]}\otimes_{(t_1,x)}\eta_{t_1,x;t_2,y}^{[t_1,+\infty)}\nu(dx,dy).
	\end{split}
	\end{equation}
	Observe that $(r^{[t_2,+\infty)})_\#M_{t_1}(\eeta)=(r^{[t_2,+\infty)})_\#\eeta$, which implies that $\{\eta_{t_2,y}^{[t_2,+\infty)}\}_{y\in Y}$ is a disintegration of $(r^{[t_2,+\infty)})_\#M_{t_1}(\eeta)$ with respect to $e_{t_2}$ and $\mu$ by Lemma \ref{lem_disintegration_push_forward} and Lemma \ref{lem_tensor_proj}. 
	Observe further that $\{\int_Y\eta^{(-\infty,t_1]}\otimes_{(t_1,x)} \eta^{[t_1,+\infty)}_{t_1,x;t_2,y}\nu_y(dx)\}_{y\in Y}$ is a disintegration of $M_{t_1}(\eeta)$ with respect to $e_{t_2}$ and $\mu$ since for every $y\in Y$, the probability measure $\int_Y\eta_{t_1,x}^{(-\infty,t_1]}\otimes_{(t_1,x)} \eta^{[t_1,+\infty)}_{t_1,x;t_2,y}\nu_y(dx)$ is concentrated on $\{\g\in \G\;:\;\g(t_2)=y\}$ so that  
	\begin{equation}
M_{t_1}(\eeta)=\int_Y\Big[	\int_Y \eta_{t_1,x}^{(-\infty,t_1]}\otimes_{(t_1,x)} \eta^{[t_1,+\infty)}_{t_1,x;t_2,y}\nu_y(dx) \Big] \mu(dy).
	\end{equation}
	Also in view of Lemma \ref{lem_tensor_proj}, and Lemma \ref{lem_disintegration_push_forward} we have that $\{\int_Y\eta_{t_1,x}^{(-\infty,t_1]}\otimes_{(t_1,x)} \eta^{[t_1,t_2]}_{t_1,x;t_2,y}\nu_y(dx)\}_{y\in Y}$ is a disintegration of $(r^{(-\infty,t_2]})_\#M_{t_1}(\eeta)$ with respect to $e_{t_2}$ and $\mu$. 
	Therefore, we get by Lemma \ref{lem_bilinearity}, Lemma \ref{lem_associativity}, and Fubini's theorem that 
	\begin{equation}
	\begin{split}
	(	M_{t_2}\circ M_{t_1})(\eeta)&=\int_{ Y}\Big[\int_{Y}\eta_{t_1,x}^{(-\infty,t_1]}\otimes_{(t_1,x)}\eta_{t_1,x;t_2,y}^{[t_1,t_2]}\nu_y(dx)\Big]\otimes_{(t_2,y)} \eta^{[t_2,+\infty)}_{t_2,y}\mu(dy)\\
	&=\int_{Y\times Y}\eta_{t_1,x}^{(-\infty,t_1]}\otimes_{(t_1,x)}\eta^{[t_1,t_2]}_{t_1,x;t_2,y}\otimes_{(t_2,y)}\eta^{[t_2,+\infty)}_{t_2,y}\nu(dx,dy).
	\end{split} 
	\end{equation}
	A similar argument shows that 
	\begin{equation} 
	(	M_{t_2}\circ M_{t_1})(\eeta)=\int_{Y\times Y}\eta_{t_1,x}^{(-\infty,t_1]}\otimes_{(t_1,x)}\eta^{[t_1,t_2]}_{t_1,x;t_2,y}\otimes_{(t_2,y)}\eta^{[t_2,+\infty)}_{t_2,y}\nu(dx,dy). 
	\end{equation} 
	The thesis follows. 
\end{proof}
As a direct consequence, we have the following lemma.
\begin{lemma}
Let $\eeta\in \mathcal{M}_\mu(\G)$.	Let $N\in \N$ and let $\s$ be a permutation of the set $ \{1,\dots,  N\}$. Then for every finite subset $\{t_1,\dots,t_N\}\subset \R$, we have 
\begin{equation}
(M_{t_{\s(1)}}\circ \dots \circ M_{t_{\s(N)}})(\eeta)=(M_{t_{1}}\circ \dots \circ M_{t_{N}})(\eeta).
\end{equation}
\end{lemma}

In view of the above lemma, we now give the following definition. 
\begin{definition} 
Let $\eeta\in \mathcal{M}_\mu(\G)$. Let $F$ be finite subset of $\R$. We define the Markovianisation operator on $\eeta$ at times $F$ denoted by $M_F(\eeta)$ through
\begin{equation} 
M_F(\eeta)=(M_{t_1}\circ \dots \circ M_{t_N})(\eeta),
\end{equation} 
where $\{t_1,\dots,t_N\}$ is any ordering of $F$, not necessarily compatible with the ordering of real numbers. 
\end{definition}

\bigskip 

\section{The disintegration map} \label{sec_disint_map}

We set ourselves in the setting of Section \ref{subsec_measure_space_cont_paths}. Let $\G$ be the space of continuous paths; let $\mathcal{M}(\G)$ be the uniform space of Radon measures on $\G$; let $\mu$ be a positive Radon measure on $Y$; and let $\mathcal{M}_\mu(\G)$ be the subuniform space consisting of every $\eeta\in\mathcal{M}(\G)$ such that $(e_t)_\#\eeta=\mu$ for every $t\in \R$.


\bigskip 

\subsection{The space $\mathcal{Z}_\mu$ of disintegrations with respect to $\mu$}

We endow $\supp\mu\subset Y$ with the induced uniformity, and consider the set of maps
\begin{equation}
\begin{split} 
{Z}_\mu:=\bigcup_{t\in\R}\Big\{&\{\nu_x\}_{x\in \supp\mu}\subset \mathcal{P}(\G)\;:\; \{\nu_x\}_{x\in \supp\mu}\text{ is a Borel family} \\
	&\quad\text{and }\;\nu_x(\{\g(t)=x\})=1 \;\forall x\in \supp\mu\Big\}.
\end{split} 
\end{equation}
Let us introduce the following relation. Letting $\{\nu^1_x\}_{x\in \supp\mu},\{\nu^2_x\}_{x\in \supp\mu}\in Z_\mu$, we shall write $$\{\nu^1_x\}_{x\in \supp\mu}\sim \{\nu^2_x\}_{x\in \supp\mu},$$ if there exists a $\mu$-negligeable set $N\subset \supp\mu$ such that for every $x\in \supp\mu\backslash N$, we have $\nu^1_x=\nu^2_x$. It can be checked that $\sim$ is an equivalence relation. We then define the set of equivalence classes on $Z_\mu$ by 
\begin{equation}
\mathcal{Z}_\mu:=\equivclass{Z_\mu }. 
\end{equation}
Elements of $\mathcal{Z}_\mu$ will be designated by bold font lowercase Greek letters for instance $\nnu$, and $\{\nu_x\}_{x\in\supp\mu}$ will then designate an element in the equivalence class of $\nnu$. 
\subsubsection{An integral on $\mathcal{Z}_\mu$} Let us introduce an integral on $\mathcal{Z}_\mu$ with respect to a test function.
The class of test functions $\mathscr{T}$ will be a countable family of nonnegative compactly supported functions of $Y$, whose integral  against $\mu$ equals one. 
Indeed $C_c(Y)$ is separable by Theorem \ref{thm_separable}. Therefore the non-empty subspace $\{\phi\in C_c(Y)\;:\;\int_Y\phi(y)\mu(dy)=1,\;\phi\geq 0\}$ is also separable, of which we let $\mathscr{T}$ be a countable and dense subset. 

Given $\phi\in \mathscr{T}$, we define the integral $$\mathcal{I}_\phi:\mathcal{Z}_\mu\ni \nnu\longmapsto \mathcal{I}_\phi(\nnu)\in \mathcal{M}(\G),$$ as the unique Radon measure in $\mathcal{M}(\G)$ -- which is canonically identified to continuous positive linear functional on $C_c(\G)$  -- satisfying
\begin{equation}
\<\mathcal{I}_\phi(\nnu),\Phi\>= \int_{Y}\int_\G\Phi(\g)\nu_x(d\g)\phi(x)\mu(dx) \qquad\forall\Phi\in C_c(\G),
\end{equation}
for some $\{\nu_x\}_{x\in \supp\mu}$ in the equivalence class of $\nnu$. 
We will also write 
\begin{equation}
\mathcal{I}_\phi(\nnu)= \int_{Y}\nu_x\phi(x)\mu(dx).
\end{equation}
\begin{lemma}
	Let $\phi\in \mathscr{T}$ and $\nnu\in \mathcal{Z}_\mu$. Then  $\mathcal{I}_\phi(\nnu)$ is well-defined and $\mathcal{I}_\phi(\nnu)\in \mathcal{P}(\G)$. 
	\end{lemma} 
\begin{proof}

	1. Let us prove that $\mathcal{I}_\phi(\nnu)$ is well-defined. Let $\{\nu_x^1\}_{x\in\supp\mu}, \{\nu_x^2\}_{x\in\supp\mu}\in Z_\mu$ such that $\nu_x^1=\nu^2_x$ for $\mu$-a.e. $x\in \supp\mu$. Then we have 
	\begin{equation}
	\int_Y\nu^1_x\phi(x)\mu(dx)=	\int_Y\nu^2_x\phi(x)\mu(dx),
	\end{equation}
	whence by Fubini's theorem, we get 
	\begin{equation}
	\int_{Y}\int_\G\Phi(\g)\nu^1_x(d\g)\phi(x)\mu(dx)=	\int_{Y}\int_\G\Phi(\g)\nu^2_x(d\g)\phi(x)\mu(dx) \qquad\forall\Phi\in C_c(\G),
	\end{equation}
	whereby $\mathcal{I}_\phi(\nnu)$ is well-defined. 
	\bigskip
	
2. Let $t\in \R$ such that $\nu_x(\{\g(t)=x\})=1$. Let $K$ be compact in $Y$ such that $\supp\phi\subset K$. Recall that $e_t^{-1}(K)$ is compact by Section \ref{subsec_space_G} So there exists $\Phi_0\in C_c(\G)$ such that $\Phi_0\equiv 1$ on $e_t^{-1}(K)$. Then since $\Phi_0\equiv 1$ on $\supp\nu_x$ for every $x\in e_t^{-1}(K)$, we have 
	\begin{equation}
	\int_Y\int_\G\Phi_0(\g)\nu_x(d\g)\phi(x)\mu(dx)=\int_Y\phi(x)\mu(dx)=1,
	\end{equation}
	and also for every $\Phi\in C_c(\G)$ such that $\Phi\equiv 0$ on $e_t^{-1}(K)$, we have 
	\begin{equation}
	\int_Y\int_\G\Phi(\g)\nu_x(d\g)\phi(x)\mu(dx)=0. 
	\end{equation}
	Therefore, in view of the Riesz-Markov-Kakutani characterisation of Radon measures, $\mathcal{I}_\phi(\nnu)(e^{-1}_t(K))=1$ and $\mathcal{I}_\phi(\nnu)(e^{-1}_t(K)^c)=0$, whence $\mathcal{I}_\phi(\nnu)\in\mathcal{P}(\G)$. 
\end{proof}

\subsubsection{A uniform structure on $\mathcal{Z}_\mu$}\label{subsec_uniform_Z_mu} Let us now endow $\mathcal{Z}_\mu$ with the intial uniformity on the set $\mathcal{Z}_\mu$  with respect to the maps $\mathcal{I}_\phi$ where $\phi$ runs through $\mathscr{T}$, i.e. $\mathcal{Z}_\mu$ is equipped with the coarsest uniform structure which makes the maps $\mathcal{I}_\phi$ uniformly continuous for every $\phi\in \mathscr{T}$.

Let us describe a fundamental system of entourages of $\mathcal{Z}_\mu$. 
For each $\phi\in \mathscr{T}$, and each entourage $V$ of $\mathcal{P}(\G)$, we define the set 
\begin{equation} \label{eqn_fund_entour}
V_{\phi}:=\Big\{(\nnu^1,\nnu^2)\in \mathcal{Z}_\mu\times \mathcal{Z}_\mu\;:\;\Big(\int_{Y} \nu^1_y\phi(y)\mu(dy), \int_{Y} \nu^2_y\phi(y)\mu(dy)\Big)\in V\Big\}.
\end{equation} 
The family  of sets $\bigcap_{i=1}^N V_{\phi_i}$
where $V$ runs through the entourages the fundamental system of entourages $\mathfrak{V}$ of $\mathcal{P}(\G)$ described in Section \ref{subsec_radon_proba}, $N$ runs through $\N$ and $\phi_i$ run in $\mathscr{T} $ forms a fundamental system of entourages of the uniformity on $\mathcal{Z}_\mu$. 

\subsection{The disintegration map $\d$}
In this section, we will define a disintegration map $\d$ on $\mathcal{M}_\mu(\G)$ and collect properties of this map.
\subsubsection{The space of continuous maps from $\R$ to $\mathcal{Z}_\mu$}\label{subsec_C_c(R;Z_mu)}
Let $\mathcal{C}_c(\R;\mathcal{Z}_\mu)$ denote the uniform space of compact convergence on the set of continuous maps from $\R$ to $\mathcal{Z}_\mu$. 
A fundamental system of entourage of $\mathcal{C}_c(\R;\mathcal{Z}_\mu)$ is then given by the sets $\WW(I;\cap_{i=1}^NV_{\phi_i})$ where $\cap_{i=1}^NV_{\phi_i}$ runs in the fundamental system of entourages described in Section \ref{subsec_uniform_Z_mu} and $I$ runs in the compacts in $\R$. 

\subsubsection{Definition of the disintegration map $\d$}
Given $\eeta\in \mathcal{M}_\mu(\G)$ and $t\in \R$, by essential uniqueness of the disintegration, there exists a unique element in $\mathcal{Z}_\mu$, which we shall denote by $\eeta_t$, and is such that every disintegration $\{\eta_{t,x}\}_{x\in \supp\mu}$ with respect to $e_t$ and $\mu$ belongs to the equivalence class $\eeta_t$.
We now define the disintegration map $\d:\mathcal{M}_\mu(\Gamma)\longrightarrow C(\R; \mathcal{Z}_\mu)$ by $\d(\eeta):\R\ni t\longmapsto \eeta_t\in \mathcal{Z}_\mu$. 

\subsubsection{Properties of the disintegration map $\d$}
The following lemma gives useful properties of the disintegration map $\d$. 
\begin{lemma}\label{lem_disint_map}
	The map $\d$ is well-defined, uniformly continuous and its image is precompact in $\mathcal{C}_c(\R;\mathcal{Z}_\mu)$. 
\end{lemma}
\begin{proof}

		To prove that $\d$ is well-defined, it suffices to show that for every $\eeta\in \mathcal{M}_\mu(\G)$, the map $$ f^{\eeta}:\R\ni t\longmapsto \eeta_t\in \mathcal{Z}_\mu$$ is continuous. 
		Recall that the uniformity on $\mathcal{Z}_\mu$ is initial with respect to the maps $\mathcal{I}_\phi$ where $\phi$ runs in $\mathscr{T}$.
	So by Proposition \ref{prop_intial}, we have that $f^{\eeta}$ is continuous, if and only if $\mathcal{I}_\phi \circ f^{\eeta}$ are continuous for each $\phi\in \mathscr{T}$, whence it suffices to prove the latter to get well-definedness.

		\bigskip 
		1. 		To prove that that $\d$ is well-defined -- and also that the image of $\d$ is precompact in $\mathcal{C}_c(\R;\mathcal{Z}_\mu)$ -- we first claim that for each $\phi\in \mathscr{T}$ and each compact interval $I\subset \R$, the family
		\begin{equation}
		\Big\{f_\phi^{\eeta,I} : I\ni t \ni \longmapsto \mathcal{I}_\phi(\eeta_t)\in \mathcal{P}(\G)\Big\}_{\eeta\in \mathcal{M}_\mu(\G)}
		\end{equation}
		is equicontinuous. 
		Let $I\subset \R$ be compact. 
		Let $t_0\in I$ and let $V\in\mathfrak{V}$ be an entourage of the fundamental system of the uniformity on $\mathcal{P}(\G)$ described in Section \ref{subsec_radon_proba}. It suffices to exhibit an open neighborhood $J\subset I$ of $t_0$ such that for every $\eeta\in \mathcal{M}_\mu(\G)$, and every $t\in J$ we have 
		\begin{equation}
	(	f^{\eeta,I}_\phi(t_0),	f^{\eeta,I}_\phi(t))\in V.
		\end{equation}
		
		Recall that $V$ has the form 
		\begin{equation}
		V=\bigcap_{i=1}^N\Big\{(\rrho^1,\rrho^2)\in \mathcal{P}(\G)\times \mathcal{P}(\G)\;:\;\Big|\int_\G \Phi_i(\g)\rrho^1(d\g)-\int_\G\Phi_i(\g)\rrho^2(d\g)\Big|<\a\Big\}. 
		\end{equation}
		
		Let $\eeta\in \mathcal{M}_\mu(\G)$ and define the measure $\nu:=(e_{t_0},e_t)_\#\eeta$ on $Y\times Y$. Define the projection maps 
		\begin{equation}
		\begin{split} 
		&q_1:Y\times Y\ni(y_1,y_2)\longmapsto y_1\in Y,\\
				&q_2:Y\times Y\ni(y_1,y_2)\longmapsto y_2\in Y.\\
		\end{split} 
		\end{equation}
		Let $\{\nu_y\}_{y\in Y}$ be a disintegration of $\nu$ with respect to $q_2$ and $\mu$. 
		Let $K$ be compact in $Y$ such that $\supp \phi \subset K$ and $e_{t_0}(\cup_{i=1}^n\supp \Phi_i)\subset K$. 
		Then $e_{t_0}^{-1}(K)$ is compact in $\G$ by Section \ref{subsec_space_G}, which from we deduce the following.

	\begin{itemize} 
	\item[(Obs)]	Observe that for each entourage $V'$ of $Y$ there exists an open neighborhood $J$ of $t_0$ such that $\supp\nu\cap K\times Y \subset V'$ whenever $t\in J$. 
		\end{itemize} 
		 Let also $\{\eta_{t_0,x;t,y}\}_{(x,y)\in Y\times Y}$ be a disintegration of $\eeta$ with respect to $e_{t_0}\times e_t$ and $\nu$. 
		 Notice that for $\mu$-a.e. $y\in Y$, we have
		 \begin{equation}
		 \eta_{t,y}=\int_{Y} \eta_{t_0,x;t,y}\nu_y(dx).
		 \end{equation}

		We then have 
		\begin{equation}
		\begin{split} 
		f^{\eeta,I}_\phi(t)=\int_{Y}\eta_{t,y}\phi(y)\mu(dy)&=\int_{Y\times Y}\eta_{t_0,x;t,y}\phi(y)\nu_y(dx)\mu(dy)\\
		&=\int_{Y\times Y}\eta_{t_0,x;t,y}\phi(y)\nu(dx,dy). 
		\end{split} 
		\end{equation}
		
		Similarly, we also have 
			\begin{equation}
		f^{\eeta,I}_\phi(t_0)=\int_{Y\times Y}\eta_{t_0,x;t,y}\phi(x)\nu(dx,dy). 
		\end{equation}

		Let us define the functions $\phi_x(x,y):=\phi(x)$ and $\phi_y(x,y):=\phi(y)$. Let $\e>0$.
		By uniform continuity of $\phi$ on $Y$, there exists an entourage $V'$ of $Y$ such that 
		\begin{equation}
		\sup_{(x,y)\in V'}|\phi_x(x,y)-\phi_y(x,y)|<\e.
		\end{equation} 
		In view of (Obs), there is an open neighborhood $J$ of $t_0$ such that for each $t\in J$, we have
		\begin{equation}
		\sup_{(x,y)\in \supp \nu\cap K\times Y} |\phi_x(x,y)-\phi_y(x,y)|<\e.
		\end{equation}
		Therefore,
		\begin{equation}
		\begin{split} 
		\Big|&\int_{\G}\Phi_i(\g) [f_\phi^{\eeta,I}(t)](d\g)-\int_{\G}\Phi_i(\g) [f_\phi^{\eeta,I}(t_0)](d\g)\Big|\\
		&=\Big|\int_{Y\times Y}\int_\G\Phi_i(\g)\eta_{t_0,x;t,y}(d\g)(\phi_x(x,y)-\phi_y(x,y))\nu(dx,dy)\Big|\\
		&=\Big|\int_{K\times Y}\int_\G\Phi_i(\g)\eta_{t_0,x;t,y}(d\g)(\phi_x(x,y)-\phi_y(x,y))\nu(dx,dy)\Big|\\
		&\leq \e \|\Phi_i\|_{C^0}\mu(K). 
		\end{split} 
		\end{equation}
		where in the second equality, we have used that $\supp \eta_{t_0,x;t,y}\cap \supp\Phi_i=\emptyset$ whenever $x\notin K$. 
		Now choose $\e>0$, and an open neighborhood $J$ of $t_0$ correspondingly such that for every $t\in J$, we have 
		\begin{equation}
		\Big|\int_{\G}\Phi_i(\g) [f_\phi^{\eeta,I}(t)](d\g)-\int_{\G}\Phi_i(\g) [f_\phi^{\eeta,I}(t_0)](d\g)\Big|<\a.
		\end{equation}
		Then for every $t\in J$, we have
		\begin{equation}
		(f_\phi^{\eeta,I}(t_0),f_\phi^{\eeta,I}(t))\in V.
		\end{equation}
		Since $t_0$ was arbitrary in $I$, and $\eeta$ was arbitrary in $\mathcal{M}_\mu(\G)$, the claim follows. This clearly implies that $\mathcal{I}_\phi\circ f^{\eeta}$ are continuous for each $\phi\in \mathscr{T}$ whereby well-definedness of $\d$ follows.

	\bigskip 
	
		2. Let us prove that the image of $\d$ is precompact in $\mathcal{C}_c(\R;\mathcal{Z}_\mu)$. By Ascoli's theorem, it suffices to show
		\begin{enumerate} 
	\item	equicontinuity for each $I\subset \R$ compact of the family of maps
	\begin{equation}
	\Big\{ f^{\eeta,I}:I\ni t\longmapsto \eeta_t\in \mathcal{Z}_\mu\}_{\eeta\in \mathcal{M}_\mu(\G)};
	\end{equation}
	\item for each $t\in \R$, the set ${\mathbf H}(t):= \{\eeta_t\;:\; \eeta\in\mathcal{M}_\mu(\G)\}$ is precompact in $\mathcal{Z}_\mu$. 
\end{enumerate}

As $\mathcal{Z}_\mu$ is equipped with the initial uniformity with respect to the maps $\mathcal{I}_\phi$ where $\phi$ runs in $\mathscr{T}$, for every $I$ compact in $\R,$ by Proposition \ref{prop_intial}, the family $\{f^{\eeta,I}\;:\; \eeta\in\mathcal{M}_\mu(\G)\}$ is equicontinuous, if and only if the families $\{\mathcal{I}_\phi\circ f^{\eeta,I}\;:\; \eeta\in\mathcal{M}_\mu(\G)\}$ are equicontinuous for every $\phi\in \mathscr{T}$, and the latter was shown in 1. of this proof.
Let us show $(ii)$.

\bigskip 

2.bis. Let us show that the family $ \{\mathcal{I}_\phi\}_{\phi\in \mathscr{T}}$ separates $\mathcal{Z}_\mu$. Let $\nnu^1,\nnu^2\in \mathcal{Z}_\mu$ such that $\mathcal{I}_{\phi}(\nnu^1)=\mathcal{I}_{\phi}(\nnu^2)$ for every $\phi\in \mathscr{T}$. Let $\mathscr{D}$ be a countable and dense subfamily of $C_c(\G)$. 
Let $\Phi\in \mathscr{D},$ and let $\{\nu^1_x\}_{x\in \supp\nu}$ and $\{\nu^2_x\}_{x\in \supp\nu}$ be in the equivalence class of $\nnu^1$ and $\nnu^2$ respectively. Observe that for every $\phi\in \mathscr{T}$, we have
\begin{equation}
\int_{Y}\phi(x)\Big[\int_\G\Phi(\g)\nu^1_{x}(d\g)\Big]\mu(dx)=\int_Y\phi(x)\Big[\int_\G\Phi(\g)\nu^2_{x}(d\g)\Big]\mu(dx).
\end{equation}
Therefore, as $\mathscr{T}$ is dense in $\{\phi\in C_c(Y):\int_Y\phi(y)\mu(dy)=1,\;\phi\geq 0\}$, by Lusin's theorem, there exists a set $N_\Phi\subset Y$ such that $\mu(N_\Phi)=0$ and for every $x\in Y\backslash N_\Phi$, we have 
\begin{equation}
\int_\G\Phi(\g)\nu_x^1(d\g)=\int_\G\Phi(\g)\nu^2_x(d\g). 
\end{equation}
Let now $N:=\cup_{\Phi\in \mathscr{D}} N_\Phi$, which is also a $\mu$-negligeable set. Then by density of $\mathscr{D}$ in $C_c(\G)$, for every $x\in Y\backslash N$, we have $\nu^1_x=\nu^2_x.$, whence $\nnu^1=\nnu^2$. 

\bigskip 
 
2.ter. Let us conclude $(ii)$. For every $\phi\in \mathscr{T}$, let $K_\phi\subset Y$ be compact such that $\supp\phi \subset K_\phi$. Then the set $e_t^{-1}(K_\phi)$ is compact in $\G$ by Section \ref{subsec_space_G}, and we have that $\mathcal{I}_\phi({\mathbf H}(t))\subset \mathcal{P}(e_t^{-1}(K_\phi))$. By Proposition \ref{prop_comp_measure}, the space $\mathcal{P}(e_t^{-1}(K_\phi))$ is vaguely compact. 
Note also that the product topological space $\Pi_{\phi\in \mathscr{T}}\mathcal{P}(e_t^{-1}(K_\phi))$ is compact by Tychonov's theorem for countable product. (This theorem can be proved within Zermelo-Fraenkel and the Axiom of Dependant Choice by \cite[Proposition 4.72]{herrlich_axiom_of_choice}). Therefore $(\mathcal{I}_\phi({\mathbf H}(t)))_{\phi\in \mathscr{T}}\subset \Pi_{\phi\in \mathscr{T}}\mathcal{P}(e_t^{-1}(K_\phi))$ is precompact. 
Observe also that the map $$\mathcal{Z}_\mu\ni \nnu \longmapsto (\mathcal{I}_\phi(\nnu))_{\phi\in \mathscr{T}}\in \Pi_{\phi\in \mathscr{T}}\mathcal{P}(e_t^{-1}(K_\phi))$$
is a homeomorphism onto its image since $\{\mathcal{I}_{\phi}\}_{\phi\in\mathscr{T}}$ separates $\mathcal{Z}_\mu$. Therefore ${\mathbf H}(t)$ is precompact in $\mathcal{Z}_\mu$, and so the image of $\d$ is precompact in $\mathcal{C}_c(\R;\mathcal{Z}_\mu)$. 

\bigskip 
		\bigskip

		3. Let us finally prove uniform continuity of $\d$. 
			Recall the fundamental system of entourages $\mathfrak{V}$ of $\mathcal{P}(\G)$ described in Section \ref{subsec_radon_proba}, and let $V\in\mathfrak{V}$. 
	Let $\cap_{i=1}^NV_{\phi_i}$ be an entourage in the fundamental system of entourages of $\mathcal{Z}_\mu$ described in Section \ref{subsec_uniform_Z_mu}. Furthermore, let $\WW(I;\cap_{i=1}^NV_{\phi_i})$ be an entourage in the fundamental system of $\mathcal{C}_c(\R;\mathcal{Z}_\mu)$ described in Section \ref{subsec_C_c(R;Z_mu)}. Let $(\eeta^1,\eeta^2)\in \WW(I;\cap_{i=1}^NV_{\phi_i})$, then we have $(\mathcal{I}_{\phi_i}(\eeta_t^1), \mathcal{I}_{\phi_i}(\eeta_t^2))\in V_{\phi_i}$ for every $t\in I$ and every $i=1,\dots, N$. 
		
		 By 1. of this proof, for every $t\in I$, there exists an open neighborhood $J_t\subset I$ containing $t$ such that for every $s\in J_t$,  every $\eeta\in \mathcal{M}_\mu(\G)$, and every $i=1,\dots,N$, we have 
		 \begin{equation}
	\Big|\int_Y\int_\G \Phi_i(\g)\phi(x)\eta_{t,x}(d\g)\mu(dx)- \int_Y\int_\G\Phi_i(\g)\phi(x)\eta_{s,x}(d\g)\mu(dx)\Big|<\a/3.
		 \end{equation}
		 Therefore, for every $t\in I$, whenever we have
		 \begin{equation}\label{eqn_condition_a/3}
		 \Big|\int_Y \int_\G \Phi_i(\g)\phi(x)\eta_{t,x}^1(d\g)\mu(dx)- \int_Y \int_\G \Phi_i(\g)\phi(x)\eta_{t,x}^2(d\g)\mu(dx)\Big|<\a/3,
		 \end{equation}
		 then for every $s\in J_t$, we have 
		 \begin{equation}
		 \Big|\int_Y\int_\G \Phi_i(\g)\phi(x)\eta^1_{s,x}(d\g)\mu(dx)- \int_Y\int_\G\Phi_i(\g)\phi(x)\eta^2_{s,x}(d\g)\mu(dx)\Big|<\a.
		 \end{equation}
		 By compactness of $I$, there exists a finite set $\{t_1,\dots,t_k\}\subset I$ such that $\{J_{t_j}\}_{1\leq j\leq k}$ is an open cover of $I$, which is such that for every $j=1,\dots,k$, every $s\in J_{t_j}$, and every $(\eeta^1,\eeta^2)\in \mathcal{M}_\mu(\G)\times \mathcal{M}_\mu(\G)$ such that \eqref{eqn_condition_a/3} holds with $t=t_j$, for every $i=1,\dots, N$, we have
		 	 \begin{equation}
		 \Big|\int_Y\int_\G \Phi_i(\g)\phi(x)\eta^1_{s,x}(d\g)\mu(dx)- \int_Y\int_\G\Phi_i(\g)\phi(x)\eta^2_{s,x}(d\g)\mu(dx)\Big|<\a.
		 \end{equation}
		 Now define the following entourage of $\mathcal{M}_\mu({\G})$ 
		\begin{equation}
		\begin{split} 
		V'&:=\bigcap_{j=1}^k\bigcap_{i=1}^N\Big\{(\eeta^1,\eeta^2)\in\mathcal{M}_\mu(\G)\times \mathcal{M}_\mu(\G)\;:\; \\
		& \Big|\int_Y\int_\G \Phi_i(\g)\phi(x)\eta^1_{t_j,x}(d\g)\mu(dx)- \int_Y\int_\G\Phi_i(\g)\phi(x)\eta^2_{t_j,x}(d\g)\mu(dx)\Big|<\a/3\Big\}. 
		\end{split} 
		\end{equation}
In view of the above argument, $(\eeta^1,\eeta^2)\in V'$ implies that for every $t\in I$, we have $(\eeta_t^1,\eeta_t^2)\in \cap_{i=1}^NV_{\phi_i}$, whence $(\d(\eeta^1),\d(\eeta^2))\in \WW(I;\cap_{i=1}^NV_{\phi_i})$. 
As $\WW(I;\cap_{i=1}^NV_{\phi_i})$ was an arbitrary entourage in the fundamental system of $\mathcal{C}_c(\R;\mathcal{Z}_\mu)$ described in Section \ref{subsec_C_c(R;Z_mu)}, the map $\d$ is uniformly continuous. 
		
\end{proof}

\subsection{The space $\mathcal{X}_\mu$ of continuous disintegrations with respect to $\mu$}
Let us define the set 
\begin{equation}
\begin{split} 
\mathcal{X}_\mu:=\bigcup_{t\in\R}\Big\{&\nnu:\supp\mu\longrightarrow \mathcal{P}(\G)  \; : \;\nnu \text{ is continuous}\\
&\text{and }\; \nu_x(\{\g\in \G\;:\;\g(t)=x\})=1\quad\forall x\in \supp \mu\Big\}. 
\end{split} 
\end{equation}
\subsubsection{A uniform structure on $\mathcal{X}_\mu$}\label{subsec_uniform_str_X_mu}
We endow $\mathcal{X}_\mu$ with the uniformity of compact convergence. 
The sets $\WW(K;V)$ where $V$ runs in the fundamental system of entourages of $\mathcal{P}(\G)$ described in Section \ref{subsec_measure_space_cont_paths} and $K$ runs in the compacts in $\supp\mu$ form a fundamental system of entourages of the uniformity on $\mathcal{X}_\mu$.  As $\mathcal{P}(\G)$ is a metrizable and separable uniform space by Lemma \ref{lem_separability}, the uniform space $\mathcal{X}_\mu$ is metrizable and separable in view of Theorem \ref{thm_separable}. 

\subsubsection{Relation between the uniformities $\mathcal{X}_\mu$ and $\mathcal{Z}_\mu$}

Let $\iota: \mathcal{X}_\mu\hookrightarrow \mathcal{Z}_\mu$ be the canonical injection mapping $\nnu\in \mathcal{X}_\mu$ to its equivalence class in $\mathcal{Z}_\mu$. 
The following lemma will be useful to establish uniform continuity of $\iota$. 
\begin{lemma}\label{lem_uniform_equicont}
	The family of of maps $\{\mathcal{I}_\phi\circ \iota : \mathcal{X}_\mu\longrightarrow \mathcal{P}(\G)\}_{\phi\in\mathscr{T}}$ is uniformly equicontinuous.
\end{lemma}
\begin{proof}
	Let $\phi\in\mathscr{T}$, let $\nnu^0\in \mathcal{X}_\mu$ and let $V$ be an entourage in the fundamental system $\mathfrak{V}$ of $\mathcal{P}(\G)$ of the form
	\begin{equation}
	V=\bigcap_{i=1}^N {\rm d}_{\Phi_i}^{-1}([0,\a]). 
	\end{equation}
	Let $K$ be a compact subset of $\supp \mu$, and let  $\WW(K;V)$ be an entourage of $\mathcal{X}_\mu$ and recall that $(\nnu^0,\nnu^1)\in \WW(K;V)$ if and only if $(\nu_x^0,\nu_x^1)\in V$ for every $x\in K$. 
	
	To prove the lemma, it now suffices to show that for every $\phi\in\mathscr{T}$, and every $(\nnu^0,\nnu^1)\in \WW(K;V)$, we have $(\mathcal{I}_\phi(\nnu^0), \mathcal{I}_\phi(\nnu^1))\in V$, whenceforth it suffices to show that for every $i=1,\dots N$, we have 
	$(\mathcal{I}_\phi(\nnu^0),\mathcal{I}_\phi(\nnu^1))\in {\rm d}_{\Phi_i}^{-1}([0,\a]). $ 
	Let $i=1,\dots,N$, $\phi\in \mathscr{T}$ and $(\nnu^0,\nnu^1)\in \WW(K;V)$. We know that for every $x\in K$, we have 
	\begin{equation}
	\Big|\int_\G\Phi_i(\g)\nu^0_x(d\g)-\int_\G\Phi_i(\g) \nu^1_x(d\g) \Big|<\a,
	\end{equation}
	which implies 
	\begin{equation}
	\Big|\int_K\int_\G\Phi_i(\g)\nu^0_x(d\g)\phi(x)\mu(dx)-\int_K\int_\G\Phi_i(\g) \nu^1_x(d\g) \phi(x)\mu(dx)\Big|<\a.
	\end{equation}
	Thus by definition of $\mathcal{I}_\phi$, we have 
	\begin{equation}
	\Big|\int_\G \Phi_i(\g)\mathcal{I}_\phi(\nnu^0)(d\g)-\int_\G \Phi_i(\g)\mathcal{I}_\phi(\nnu^1)(d\g)\Big|<\a,
	\end{equation}
	which we can recast as 
	\begin{equation}
	(\mathcal{I}_\phi(\nnu^0),\mathcal{I}_\phi(\nnu^1))\in {\rm d}^{-1}_{\Phi_i}([0,\a]). 
	\end{equation}
	The thesis follows since $\phi$ was arbitrary in $\mathscr{T}$. 
\end{proof}

Let us now prove continuity of the canonical injection of $\mathcal{X}_\mu$ into $\mathcal{Z}_\mu$. 

\begin{lemma}\label{lem_canonic_inject_cont}
	The canonical injection $\iota:\mathcal{X}_\mu\hookrightarrow \mathcal{Z}_\mu$ is uniformly continuous. 
\end{lemma}
\begin{proof}
	Recall that the uniformity on $\mathcal{Z}_\mu$ is initial with respect to $\mathcal{I}_\phi$ where $\phi$ runs in $\mathscr{T}$. 
	Therefore, by Proposition \ref{prop_intial}, the canonical injection $\iota:\mathcal{X}_\mu\hookrightarrow \mathcal{Z}_\mu$ is uniformly continuous, if and only $\mathcal{I}_{\phi}\circ \iota$ is uniformly continuous for each $\phi\in \mathscr{T}$, and by Lemma \ref{lem_uniform_equicont}, the latter holds, which shows that $\iota$ is uniformly continuous. 
\end{proof} 

\subsubsection{The uniformity on $\mathcal{C}_c(\R;\mathcal{X}_\mu)$}\label{subsec_uniform_C_c(R;X_mu)}
We endow $\mathcal{C}_c(\R;\mathcal{X}_\mu)$ with the uniformity of compact convergence. The sets $\WW(K;V)$ where $V$ runs in the fundamental system of entourages of $\mathcal{X}_\mu$ described in Section \ref{subsec_uniform_str_X_mu} and $K$ runs in the compacts in $\supp\mu$ form a fundamental system of entourages of the uniformity on $\mathcal{X}_\mu$. As $\mathcal{X}_\mu$ is metrizable and separable, in view of Theorem \ref{thm_separable}, the uniformity on $\mathcal{C}_c(\R;\mathcal{X}_\mu)$ is metrizable and separable.
\subsection{Propagating invariance under the Markov opertor}
Let $t\in \R$ and $x\in Y$. 
Recall that $r^I$ is the restriction map from $\mathcal{M}(\G)$ to $\mathcal{M}(\G\lfloor_I)$ for every interval $I\subset \R$.
We define $\G_{t,x}:=\{\g\in \G\;:\;\g(t)=x\}$ which is endowed with the induced uniform structure. By a slight abuse of notation, the induced restriction map from $\G_{t,x}$ to $\G_{t,x}\lfloor_I$ is also designated by $r^I$, and for every $\nu\in \mathcal{P}(\G_{t,x})$, we write $\nu^I:=(r^I)_\#\nu$. We then define the Markov opertor $M_{t,x}$ from $\mathcal{P}(\G_{t,x})$ to $\mathcal{P}(\G_{t,x})$ by 
$M_{t,x}(\nu)=\nu^{(-\infty,t]}\otimes_{(t,x)}\nu^{[t,+\infty)}$ where we recall that, in the sense of Proposition \ref{prop_tensor}, the tensor product here is through the map $f_t:\G_{t,x}\longrightarrow\G_{t,x}\lfloor_{(-\infty,t]}\times\G_{t,x}\lfloor_{[t,+\infty)}$ given by $f_t(\g)=(r^{(-\infty,t]}\g, r^{[t,+\infty)}\g)$, which is an isomorphism of uniform spaces. 

\begin{lemma}\label{lem_markov_op_cont}
	Let $(\nu_n)_{n\in\N}$ be a sequence in $\mathcal{P}(\G_{t,x})$ and let $\nu\in \mathcal{P}(\G_{t,x})$. Then $M_{t,x}\nu_n$ converges vaguely to $M_{t,x}\nu$ as $n\to+\infty$, if and only if $\nu_n$ converges vaguely to $\nu$ as $n\to+\infty$. 
\end{lemma}
\begin{proof}
	Firstly by Proposition \ref{prop_comp_measure} and Section \ref{subsec_space_G}, we have that $\mathcal{P}(\G_{t,x})$ is vaguely compact so it suffices to check vague convergence in $\mathcal{P}(\G_{t,x})$. 
	From the characterisation of the tensor product, for every $\bar\nu\in \mathcal{P}(\G_{t,x})$ we have 
	\begin{equation}
	(f_t)_\#M_{t,x}\bar\nu=\bar\nu^{(-\infty,t]}\times \bar\nu^{[t,+\infty)},
	\end{equation}
	which implies that
		\begin{equation}
	M_{t,x}\bar\nu=(f_t)^\#\bar\nu^{(-\infty,t]}\times \bar\nu^{[t,+\infty)}. 
	\end{equation}
	Let $\Phi\in C_c(\G_{t,x})$ and notice that 
	\begin{equation}
	\int_{\G_{t,x}} \Phi(\g)M_{t,x}\nu_n(d\g)= \int_{\G^{[t,+\infty)}_{t,x}}\int_{\G^{(-\infty,t]}_{t,x}}\Phi(f_t^{-1}(\g_1,\g_2))\nu_n^{(-\infty,t]}(d\g_1)\nu_n^{[t,+\infty)}(d\g_2). 
	\end{equation}
	Since $f_t=r^{(-\infty,t]}\times r^{[t,+\infty)}$ is an isomorphism of uniform spaces, the pushforward map $(f_t)_\#$ is an homeomorphism of $\mathcal{P}(\G_{t,x})$ onto $\mathcal{P}(\G^{(-\infty,t]}_{t,x}\times \G^{[t,+\infty)}_{t,x})$, whereby $\nu_n$ converges vaguely to $\nu$, if and only if $\nu^{(-\infty,t]}_n\times \nu_n^{[t,+\infty)}$ converges vaguely to $\nu^{(-\infty,t]}\times \nu^{[t,+\infty)}$, whence by Fubini's theorem
	\begin{equation}
	\begin{split} 
	&\lim_{n\to+\infty}	\int_{\G_{t,x}} \Phi(\g)M_{t,x}\nu_n(d\g)\\
	&=\lim_{n\to+\infty} \int_{\G_{t,x}\lfloor_{[t,+\infty)}}\int_{\G_{t,x}\lfloor_{(-\infty,t]}}\Phi(f_t^{-1}(\g_1,\g_2))\nu_n^{(-\infty,t]}(d\g_1)\nu_n^{[t,+\infty)}(d\g_2)\\
	&= \int_{\G_{t,x}\lfloor_{[t,+\infty)}}\int_{\G_{t,x}\lfloor_{(-\infty,t]}}\Phi(f_t^{-1}(\g_1,\g_2))\nu^{(-\infty,t]}(d\g_1)\nu^{[t,+\infty)}(d\g_2)\\
	&= \int_{\G_{t,x}}\Phi(\g) M_{t,x}\nu(d\g). 
	\end{split} 
	\end{equation}
	Since $\Phi$ was arbitrary, it follows that $\nu_n^{(-\infty,t]}\times \nu_n^{[t,+\infty)}$ converges vaguely to $\nu^{(-\infty,t]}\times \nu^{[t,+\infty)}$ as $n\to+\infty$, if and only if $M_{t,x}\nu_n$ converges vaguely to $M_{t,x}\nu$ as $n\to+\infty$ and the thesis follows. 
	\end{proof} 

For every $h\in \R$ and every interval $I\subset\R$, denote $I-h:=\{t-h\;:\;t\in I\}$, and define the maps $b_h$ from $\G\lfloor_I$ to $\G\lfloor_{I-h}$ by $b_h(\g(\cdot))=\g(\cdot-h). $ Note that in view of Section \ref{subsec_space_G}, the maps $b_h$ are well-defined. By a slight abuse of notation, the domain and codomain are implicitly designated by the symbol $b_h$, and for every $s\in I$, the induced map from $\G_{s,x}\lfloor_I$ to $\G_{s-h,x}\lfloor_{I-h}$ is also designated by the symbol $b_h$. 
\begin{itemize}
	\item[(Obs-$\Phi$)] For every $\Phi\in C_c(\G)$, the set $\{\Phi\circ b_h\}_{h\in (-1,1)}$ is precompact in $C_c(\G)$. 
\end{itemize}
Let us now state the following lemma. 
\begin{lemma}\label{lem_markov_limit}
Let $(t_n)_{n\in\N}$ be a sequence in $\R$, and let $(\nu_n)_{n\in\N}$ be a sequence in $\mathcal{P}(\G)$ concentrated on $\G_{t_n,x}$ for every $n\in\N$. Assume that for some $t\in \R$ and some $\nu\in \mathcal{P}(\G)$, we have $t_n\to t$ as $n\to+\infty$ and $\nu_n$ converges vaguely to $ \nu$ as $n\to+\infty$, and that for every $n\in\N$, we have $M_{t_n,x}\nu_{n}=\nu_{n}$. Then $\nu$ is invariant under $M_{t,x}$, i.e. $M_{t,x}\nu=\nu$. 
\end{lemma}
\begin{proof}

1.
We claim that $(b_h)_\#M_{s,x}\nu =M_{s+h,x} (b_h)_\#\nu$. By the characterising property of the tensor product in Proposition \ref{prop_tensor}, we have that 
\begin{equation}
(f_{s+h})_\#M_{s+h,x} (b_h)_\#\nu = ((b_h)_\#\nu )^{(-\infty,s+h]}\times ((b_h)_\#\nu )^{[s+h,+\infty)}. 
\end{equation}
	Also we have $f_{s+h}\circ b_h=(b_h\times b_h)\circ f_s$, therefore
	\begin{equation}
	\begin{split} 
(	f_{s+h})_\#(b_h)_\#M_{s,x}\nu&=(b_h\times b_h)_\# (f_s)_\#M_{s,x}\nu\\
&=(b_h\times b_h)_\#\nu^{(-\infty,s]}\times \nu^{[s,+\infty)}\\
&=((b_h)_\#\nu)^{(-\infty,s+h]}\times ((b_h)_\#\nu)^{[s+h,+\infty)}. 
\end{split} 
	\end{equation}
	The claim follows. 
	
	\bigskip 
	
	2. We claim that for every $h\in \R$, the pushforward map $(b_h)_\#$ from $\mathcal{P}(\G)$ to $\mathcal{P}(\G)$ is sequential continuous. Let $(\bar\nu_n)_{n\in\N}$ be a sequence in $\mathcal{P}(\G)$ such that $\bar\nu_n$ converges vaguely to $\bar\nu$ in $\mathcal{P}(\G)$, and let $\Phi \in C_c(\G)$. Observe that $\Phi\circ b_h\in C_c(\G)$, whence
	\begin{equation}
	\lim_{n\to+\infty}\int_\G \Phi(b_h\g)\bar\nu_n(d\g)=\int_\G \Phi(b_h\g)\bar\nu(d\g). 
	\end{equation}
	This proves the claim as $\Phi$ was arbitrary. 
	\bigskip 
	
	2.bis. Let $\bar\nu\in \mathcal{P}(\G)$. We claim that $(b_h)_\#\bar\nu$ converges vaguely to $\bar\nu$ as $h\to 0$. Let $\Phi \in C_c(\G)$ and notice that $\Phi\circ b_h$ converges to $\Phi$ in $C_c(\G)$ as $h\to 0$ by (Obs-$\Phi$).
	Therefore 
	\begin{equation}
	\lim_{h\to 0} \int_\G \Phi(b_h \g)\bar\nu(d\g)= \int_\G \Phi( \g)\bar\nu(d\g),
	\end{equation}
	which proves the claim since $\Phi$ was arbitrary. 
	
	\bigskip 
	
	3. We claim that $(b_{t-t_n})_\#\nu_n$ converges vaguely to $\nu$.
	Observe that for every $\g\in \G$, we have that $b_{t-t_n}\g$ converges to $\g$ in $\G$ as $n\to+\infty$. 
	 Let $\Phi \in C_c(\G)$, and observe that by (Obs-$\Phi$), we have that  $\Phi(b_{t-t_n}\g)$ converges to $\Phi(\g)$ in $C_c(\G)$ as $n\to+\infty$. Since $\nu_n$ converges vaguely to $\nu$ as $n\to+\infty, $ we have
	\begin{equation}
\lim_{n\to+\infty}	\int_\G\Phi (b_{t-t_n}\g)\nu_n(d\g)=\int_\G \Phi(\g)\nu(d\g). 
	\end{equation}
	As $\Phi$ was arbitrary, the claim follows. 
	\bigskip

	3.bis. In view of Lemma \ref{lem_markov_op_cont}, we have that $M_t(b_{t-t_n})_\#\nu_n$ converges vaguely to $M_t\nu$ as $n\to+\infty$, and by 1. of this proof, $(b_{t-t_n})_\#M_{t_n}\nu_n$ converges vaguely to $M_t\nu$ as $n\to+\infty$. By 2.bis. of this proof, we have that $(b_{t_n-t})_\#M_t\nu$ converges vaguely to $M_t\nu$ as $n\to+\infty$. Also by 2. of this proof, $(b_{h})_\#(b_{t-t_n})_\# M_{t_n}\nu_n$ converges vaguely to $(b_{h})_\#M_t\nu$ as $n\to+\infty$. Therefore, setting $h=t_n-t$, we find by 2.bis. of this proof that $M_{t_n}\nu_n$ converges vaguely to $M_t\nu$ as $n\to+\infty$, whence $M_t\nu=\lim_{n\to+\infty}M_{t_n}\nu_n=\lim_{n\to+\infty}\nu_n =\nu.$ This concludes the proof. 

\end{proof}

\subsection{Proof of Theorem \ref{thm_main_strong_markov}}
By a slight abuse of notation, for every $\eeta\in \mathcal{M}_\mu(\G)$ every $t\in \R$, whenever there exists a disintegration of $\eeta$ with respect to $e_t$ and $\mu$ which lies in $\mathcal{X}_\mu$, we will designate this disintegration as well as the corresponding equivalence class in $\mathcal{Z}_\mu$ by $\eeta_t$. 
Define the set 
\begin{equation} 
\begin{split} 
\mathcal{Y}_\mu:=\Big\{&\eeta\in \mathcal{M}_\mu(\G)\;:\; \eeta_t\in \mathcal{X}_\mu \text{ for every }t\in \R\\
&\qquad\text{ and }\; \R\ni t\longmapsto \eeta_t\in \mathcal{X}_\mu \text{ is continuous}\Big\},
\end{split} 
\end{equation} 
which we endow with the induced uniformity. 
We then define the disintegration map $\d':\mathcal{Y}_\mu\longrightarrow \mathcal{C}_c(\R;\mathcal{X}_\mu)$ by $\d'(\eeta):\R\ni t\longmapsto \eeta_t\in \mathcal{X}_\mu$, which is clearly well-defined. 

We can now define elements in $\mathcal{M}_\mu(\G)$, which satisfy the strong Markov property.
\begin{definition}\label{def_strong_markov}
	We shall say that $\eeta\in \mathcal{Y}_\mu$ satisfies the strong Markov property, if for every $(t,x)\in \R\times \supp\mu$, we have 
\begin{equation}
\eta_{t,x}=\eta_{t,x}^{(-\infty,t]}\otimes_{(t,x)}\eta_{t,x}^{[t,+\infty)}. 
\end{equation}
	\end{definition}


\begin{lemma}\label{lem_d'}
	The map $\d':\mathcal{Y}_\mu\longrightarrow \mathcal{C}_c(\R;\mathcal{X}_\mu)$ is uniformly continuous and injective. 
\end{lemma}
\begin{proof}
	Denote by $\iota:\mathcal{Y}_\mu \hookrightarrow \mathcal{M}_\mu(\G)$ the canonical injection. 
Observe that $\d'=\d\circ \iota$, whence uniform continuity follows from the uniform continuity of $\d$ and $\iota$. Let us prove injectivity. Let $\eeta^1,\eeta^2\in \mathcal{Y}_\mu$ and suppose that $(\d\circ\iota)(\eeta^1)=(\d\circ\iota)(\eeta^2)$. Then for some $t\in \R,$ we have 
	$\eeta_t^1=\eeta_t^2$, which implies that by essential uniqueness of the disintegration that $\eta^1_{t,x}=\eta^2_{t,x}$ for $\mu$-a.e. $x\in \supp\mu $. Therefore
	\begin{equation}
	\eeta^1=\int_Y\eta^1_{t,x}\mu(dx)=\int_Y\eta^2_{t,x}\mu(dx)=\eeta^2,
	\end{equation}
	whereby injectivity follows. 
\end{proof}
\begin{definition}\label{def_markov_regular}
	We shall say that $\eeta\in \mathcal{Y}_\mu$ is Markov regular, if for every compact $I$ in $\R$, there exists a compact $K$ in $\mathcal{X}_\mu$ such that $(M_F(\eeta))_t\in K$ for every finite subset $F$ of $\R$. 
\end{definition}
The following lemma shows that the image by $\d'$ of Markov regular elements of $\mathcal{Y}_\mu$ enjoy additional regularity. 
\begin{lemma}\label{lem_compact_valued} 
	Let $\eeta\in\mathcal{Y}_\mu$ be Markov regular, and let $I$ be compact in $\R$. Then there exists a compact $K\subset \mathcal{X}_\mu$ such that for every finite subset $F$ of $\R$, we have $\d'(M_F(\eeta ))\lfloor_I\in C(I;K)$. 
\end{lemma}
\begin{proof}
By definition of Markov regularity, for every $t\in I$, there exists a compact $K\subset \mathcal{X}_\mu$ such that $(M_F(\eeta))_t\in K$ for every finite subset $F$ of $\R$.  
	Define the canonical injection $\iota_K:K\hookrightarrow \mathcal{X}_\mu$ and recall the injection $\iota:\mathcal{X}_\mu\hookrightarrow \mathcal{Z}_\mu$, which is continuous by Lemma \ref{lem_canonic_inject_cont}.
	Since $\d$ is well-defined by Lemma \ref{lem_disint_map}, the following map is  continuous
	\begin{equation}
	f: I\ni t\longmapsto (M_F(\eeta))_t\in \mathcal{Z}_\mu. 
	\end{equation}
A priori, the subspace topology on $K$ and the initial topology on $K$ with respect to $\iota\circ \iota_K$ need not coincide, but since $\iota\circ \iota_K$ is injective, the coarsest topology on $K$, which makes $\iota\circ \iota_K$ continuous, is characterised by the fact that $\iota\circ \iota_K$ is an homeomorphism onto its image. So the subspace topology on $K$ is homeomorphic to the initial topology on $K$ with respect to $\iota\circ \iota_K$. 
Thus, there exists a unique continuous map 
	\begin{equation}
	g:I\ni t\longmapsto (M_F(\eeta))_t\in K,
	\end{equation} 
	such that $(\iota\circ \iota_K) \circ g =  f$, whence $g=\d'(M_F(\eeta ))\lfloor_I\in C(I;K)$.
\end{proof}

A \emph{subset ordering} of a countable and infinite set $D$ is a map from the natural numbers to the power set of $D$ i.e. $\N\ni N\longmapsto D_N\in \mathscr{P}(D)$, such that for every $N\in\N$, we have $D_{N+1}\backslash D_N$ is a singleton and $D_1$ is a singleton. 

The following lemma will be essential to show that all elements of the Markov hull of a Markov regular element of $\mathcal{Y}_\mu$ satisfy the strong Markov property. 

\begin{lemma}\label{lem_precompact}
	Let $\eeta\in\mathcal{Y}_\mu$ be Markov regular, and let $D$ be a countable, dense subset of $\R$. Let $(D_N)_{N\in\N}$ be a set ordering of $D$. Then the sequence
	$(\d' (M_{D_N}(\eeta)))_{N\in\N}$ is precompact in the uniform space of compact convergence $\mathcal{C}_c(\R;\mathcal{X}_\mu)$.
\end{lemma}

\begin{proof}
	As $\mathcal{C}_c(\R;\mathcal{X}_\mu)$ is metrizable and separable by Section \ref{subsec_uniform_C_c(R;X_mu)}, we have that
 sequential precompactness and precompactness are equivalent in $\mathcal{C}_c(\R;\mathcal{X}_\mu)$. It thus suffices to show that for every compact interval $I$ in $\R$, the sequence  $(\delta(M_{D_N}(\eeta))\lfloor_I)_{N\in\N}$ is compact in
	the uniform space $\mathcal{C}_u(I;\mathcal{X}_\mu)$ of uniform convergence of continuous maps from $I$ to $\mathcal{X}_\mu$. 
	
		Let $I$ be a compact interval in $\R$.  In view of Lemma \ref{lem_compact_valued}, there exists a compact $K\subset \mathcal{X}_\mu$ such that $\delta'(M_{D_N}(\eeta))\lfloor_I$ is continuous from $I$ to $K$ for every $N\in\N$. By Lemma \ref{lem_canonic_inject_cont}, the canonical injection $\iota: \mathcal{X}_\mu\hookrightarrow \mathcal{Z}_\mu$ is uniformly continuous, whence $\iota(K)$ is compact in $\mathcal{Z}_\mu$. 
	In view of Lemma \ref{lem_disint_map}, the sequence  $(\d(M_{D_N}(\eeta))\lfloor_I)_{N\in\N}$ is precompact in $\mathcal{C}_u(I;\mathcal{Z}_\mu)$, and since $(M_{D_N}(\eeta))_t$ lies in $\iota(K)$ for every $t\in I$, the sequence $(\d(M_{D_N}(\eeta))\lfloor_I)_{N\in\N}$ is also precompact in $\mathcal{C}_u(I;\iota(K))$. Since $K$ and $\iota(K)$ are isomorphic as uniform spaces by Theorem \ref{thm_unique_uniform_structure_compact}, $(\d'(M_{D_N}(\eeta))\lfloor_I)_{N\in\N}$ is precompact in $\mathcal{C}_u(I;K)$, whence it is also precompact in $\mathcal{C}_u(I;\mathcal{X}_\mu)$. 
		The claim follows since $I$ was an arbitrary compact in $\R$. 
	
\end{proof}
We now define the $D$-Markov hull of an element in $\mathcal{M}_\mu(\G)$. 
\begin{definition}\label{def_markov_D-hull}
		Let $\eeta\in \mathcal{M}_\mu(\G)$ and $D$ be a dense, countable subset of $\R$. We shall call Markov $D$-hull of $\eeta$ the set consisting of all vague accumulation points in $\mathcal{M}_\mu(\G)$ of sequences of the form $(M_{D_N}(\eeta))_{N\in\N}$ where $(D_N)_{N\in\N}$ run in the class of all set ordering of $D$. We shall denote the Markov $D$-hull of $\eeta$ by $\mathfrak{M}_D(\eeta)$, and we can write it as
		\begin{equation}
		\begin{split} 
		\mathfrak{M}_D(\eeta)=\bigcup_{(D_N)_{N\in\N}\in Ord(D)} &\Big\{\bar\eeta\in\mathcal{M}_\mu(\G)\;:\;
		\bar\eeta \text{ is a vague accumulation point} \\
		&\text{of the sequence }(M_{D_N}(\eeta) )_{N\in\N}\Big\},
		\end{split} 
		\end{equation}
		where $Ord(D)$ denotes the class of all orderings of the set $D$.
\end{definition}
We then define the Markov hull of an element in $\mathcal{M}_\mu(\G)$.  
\begin{definition}\label{def_markov_hull}
	Let $\eeta\in \mathcal{M}_\mu(\G)$. We shall call Markov hull of $\eeta$ the subset of $\mathcal{M}_\mu(\G)$ formed as the union of the Markov $D$-hulls of $\eeta$ over the family of dense and countable subsets $D\subset \R$. We shall denote the Markov hull of $\eeta$ by  $\mathfrak{M}(\eeta)$, and we can write it as
	\begin{equation}
	\mathfrak{M}(\eeta)=\bigcup_{D\in\mathfrak{D}}\mathfrak{M}_D(\eeta),
	\end{equation}
	where $\mathfrak{D}\subset \mathscr{P}(\R)$ is the family of subsets of $\R$ which are dense and countable. 
\end{definition}

	Recalling the canonical injection $\iota:\mathcal{X}_\mu\hookrightarrow \mathcal{Z}_\mu$, we define the induced injection $$\iota^*:\mathcal{C}_c(\R;\mathcal{X}_\mu)\ni \{\nnu_t\}_{t\in\R }\longmapsto \{\iota\nnu_t\}_{t\in\R}\in \mathcal{C}_c(\R;\mathcal{Z}_\mu).$$
	\begin{lemma}\label{lem_iota*}
		The map $\iota^*$ is well-defined and continuous. Furthermore $\iota^*\circ \d'=\d\lfloor_{\mathcal{Y}_\mu}$. 
	\end{lemma}
\begin{proof}
	To see that $\iota^*$ is well-defined, notice that the map $\R\ni t\longmapsto \iota(\nnu_t)\in \mathcal{Z}_\mu$ is continuous, since $\R\ni t\longmapsto \nnu_t\in \mathcal{X}_\mu$ is continuous and $\iota:\mathcal{X}_\mu\hookrightarrow \mathcal{Z}_\mu$ is also continuous. As $\mathcal{C}_c(\R;\mathcal{X}_\mu)$ and $\mathcal{C}_c(\R;\mathcal{Z}_\mu)$ are metrizable, it suffices to verify the sequential characterisation of continuity. Let $(\{\nnu^n_t\}_{t\in\R})_{n\in\N}$ be a sequence converging in $\mathcal{C}_c(\R;\mathcal{X}_\mu)$ to $\{\nnu_t\}_{t\in\R}$. Let $I$ be a compact in $\R$, and we then have that $\{\nnu^n_t\}_{t\in I}$ converges in $\mathcal{C}_u(I;\mathcal{X}_\mu)$ to $\{\nnu_t\}_{t\in I}$ as $n\to+\infty$. 
	Since $\iota$ is uniformly continuous, this implies that $\{\iota \nnu^n_t\}_{t\in I}$ converges in $\mathcal{C}_u(I;\mathcal{Z}_\mu)$ to $\{\iota \nnu_t\}_{t\in I}$ as $n\to+\infty$, whence $\{\iota\nnu^n_t\}_{t\in \R}$ converges in $\mathcal{C}_c(\R;\mathcal{Z}_\mu)$ to $\{\iota\nnu_t\}_{t\in \R}$ as $n\to+\infty$.	The identity $\iota^*\circ \d'=\d\lfloor_{\mathcal{Y}_\mu}$ can be checked directly.
\end{proof}
We now prove that every element of the Markov hull of a Markov regular element in $\mathcal{Y}_\mu$ satisfies the strong Markov property. 
\begin{lemma}\label{lem_markov_all_times}
Let $\eeta\in\mathcal{Y}_\mu$ be Markov regular. Then every element in the Markov hull of $\eeta$ satisfies the strong Markov property.
\end{lemma}
\begin{proof} 
	Let $\bar\eeta\in\mathcal{M}_\mu(\G)$ be in the Markov hull of $\eeta$. Then there exists a countable, dense subset $D$ of $\R$ and a set ordering $(D_N)_{N\in\N}$ of $D$ such that, modulo a subsequence $M_{D_N}(\eeta)$ converges vaguely to $\bar\eeta$ as $N\to+\infty$. By Lemma \ref{lem_disint_map}, we also have that $\d(M_{D_N}(\eeta))$ converges to $\d(\bar\eeta)$ in $\mathcal{C}_c(\R;\mathcal{Z}_\mu)$ as $N\to+\infty$. 
	
	\bigskip 
	
	1. We claim that $\d'(M_{D_{N}}(\eeta))$ converges to $\d'(\bar\eeta)$ in $\mathcal{C}_c(\R;\mathcal{X}_\mu)$ as $N\to+\infty$. 
		Let $\xi:\N\to \N$ be an increasing map. Recall that the uniform space $\mathcal{C}_c(\R;\mathcal{X}_\mu)$ is metrizable by Theorem \ref{thm_separable}, so by Lemma \ref{lem_precompact}, there exists another increasing map $\z:\N\to\N$ such that 
	$\d'(M_{D_{(\zeta\circ\xi)(N)}}(\eeta))$ converges in $\mathcal{C}_c(\R;\mathcal{X}_\mu)$ as $N\to+\infty$ to some element $\rrho=\{\rrho_t\}_{t\in\R}.$ Observe that $\iota^*\rrho\in \mathcal{C}_c(\R;\mathcal{Z}_\mu)$ and that $\d(M_{D_{(\zeta\circ\xi)(N)}}\eeta)$ converges in $\mathcal{C}_c(\R;\mathcal{Z}_\mu)$ to $\iota^*\rrho$ as $N\to+\infty$ by Lemma \ref{lem_iota*}. Therefore, we have $\iota^*\rrho=\d(\bar\eeta),$ whence $\rrho=\d'(\bar\eeta)$. By the subsubsequence lemma, it follows that $\d'(M_{D_{N}}(\eeta))$ converges to $\rrho$ in $\mathcal{C}_c(\R;\mathcal{X}_\mu)$ as $N\to+\infty$. 
	 
	\bigskip 
	2. Let us now show that $\bar\eeta$ satisfies the strong Markov property. 
	Let $K$ be a compact in $\supp\mu$, let $(V_k)_{k\in\N}$ be a decreasing fundamental system of entourages  of $\mathcal{P}(\G)$, i.e. $V_{k+1}\subset V_k$ for every $k\in\N$. Let $k\in\N$, let $\WW(K;V_k)$ be an entourage of $\mathcal{X}_\mu,$ and let $t\in \R$. Then there exists an entourage $\WW_k'$ of $\mathcal{X}_\mu$ such that $\WW_k'\circ \WW_k'\subset \WW(K;V_k)$ by Axiom $(U_3)$ of Definition \ref{def_uniform_structure}. Also there exists $N_k\in\N$  and $t_k\in D_{N_k}\cap [t-1,t+1]$ such that:
	\begin{itemize} 
			 \item  $t_k\to t$ as $k\to+\infty$ thanks to density of $D$ in $[t-1,t+1]$;
			 		 \item  $(\bar\eeta_{t_k},\bar\eeta_t)\in \WW_k'$ thanks to uniform continuity of $[t-1,t+1]\ni t\longmapsto \eeta_t\in \mathcal{X}_\mu$;
	\item 	 $((M_{D_{N_k}}(\eeta))_{t_k},\bar\eeta_{t_k})\in \WW_k'$ thanks to convergence of $\d'(M_{D_N}(\eeta))$ to $\d'(\bar\eeta)$ in $\mathcal{C}_u([t-1,t+1];\mathcal{X}_\mu)$. 
		\end{itemize}
	
	\bigskip 
	
	Therefore, we have $((M_{D_{N_k}}(\eeta))_{t_k},\bar\eeta_{t_k})\in \WW(K;V_k), $ which implies that for every $x\in K$, we have $(M_{D_{N_k}}(\eeta))_{t_k,x},\bar\eta_{t,x})\in V_k$. Setting $\nu_k=(M_{D_{N_k}}(\eeta))_{t_k,x}$ and $\nu =\bar\eta_{t,x}$, the hypothesis of Lemma \ref{lem_markov_limit} are satisfied, so that we have $M_{t,x}\bar\eta_{t,x}=\bar\eta_{t,x}$ for every $x\in K$. As $K$ and $t$ were arbitrary, this proves the thesis. 
 	\end{proof}

Let us now conclude the proof of Theorem \ref{thm_main_strong_markov}. 
\begin{proof}[Proof of Theorem \ref{thm_main_strong_markov}]
Let $\eeta$ be Markov regular, and let $D\subset \R$ be dense and countable, and let $(D_N)_{N\in\N}$ be a set ordering of $D$. Note that $M_{D_N}(\eeta)\in \mathcal{M}_\mu(\G)$ for every $N\in\N$, whence by Lemma \ref{lem_mu_invariant_space}, $(M_{D_N}(\eeta))_{N\in\N}$ has at least one accumulation point in $\mathcal{M}_\mu(\G)$, 
which shows that the Markov hull of $\eeta$ is non-empty. Lemma \ref{lem_markov_all_times} then shows that every element of the Markov hull of $\eeta$ satisfies the strong Markov property. 
	\end{proof} 

\bigskip 

\section{Translation invariant measures} \label{sec_translation_inv}
Let $(Y,\cdot)$ be a locally compact Polish group. Let $\mu$ be the left Haar measure on $Y$. $\mu$ is uniquely determined up to a multiplicative constant. 
Recall that  $\mathcal{C}_c(\R; Y)$ is the uniform space of compact convergence. 
Let $\G$ be as in Section \ref{subsec_space_G}. 
Define the left action of $Y$ on $\mathcal{C}_c(\R; Y)$ by $Y \times \mathcal{C}_c(\R; Y)\ni (y,\g)\longmapsto y\cdot \g \in\mathcal{C}_c(\R; Y)$. Assume additionally that $\G$ is such that the induced left action of $Y$ on $\G$ is transitive. 
(All this section will study left Haar measures and left actions $Y$, but can be modified straightforwardly to address the case of right Haar measures and right actions on $Y$.)

\subsection{Left translation of paths}
We define the left translation of paths $$l_y:\mathcal{C}_c(\R;Y)\ni \g\longmapsto y\cdot \g\in \mathcal{C}_c(\R;Y),$$
and by transitivity of the action of $Y$ on $\G$, the restriction $l_y\lfloor_\G$ of $l_y$ to $\G$ is bijective onto $\G$. 
Denote by $(l_y)_\#$ the induced pushforward map from $\mathcal{M}(\G)$ to $\mathcal{M}(\G)$, and notice that $(l_y)_\#\mathcal{M}_\mu(\G)\subset \mathcal{M}_\mu(\G).$
\begin{lemma}\label{lem_cont_ly}
	Let $y\in Y$. Then the map $(l_y)_\#:\mathcal{P}(\G)\to \mathcal{P}(\G)$ is vaguely continuous. 
\end{lemma}
\begin{proof}
	It suffices to check the sequential characterisation of continuity since $\mathcal{P}(\G)$ is metrizable by Lemma \ref{lem_separability}. 
	Let $(\nu_n)_{n\in\N}$ be a sequence in $\mathcal{P}(\G)$ converging vaguely to $\nu\in \mathcal{P}(\G)$, and let $\Phi\in C_c(\G)$. Observe that $\Phi\circ l_y\in C_c(\G)$, whence
	\begin{equation}
	\lim_{n\to+\infty}\int_\G \Phi(\g)(l_y)_\#\nu_n(d\g)=	\lim_{n\to+\infty}\int_\G \Phi(l_y\g)\nu_n(d\g)=\int_\G\Phi(l_y\g)\nu(d\g). 
	\end{equation}
	Thus $(l_y)_\# \nu_n$ converges vaguely to $(l_y)_\#\nu$ as $n\to+\infty$. This proves the thesis. 
\end{proof}

\begin{lemma}\label{lem_conv_l_x}
	Let $\g\in \G$. Then $l_x\g$ converges to $l_y\g$ in $\G$ as $x\to y$. 
\end{lemma}
\begin{proof}
	Without loss of generality, we can take $y=e.$ By continuity of the group multiplication, for every $t\in \R$, we have that $x\cdot \g(t)$ converges to $\g(t)$ in $Y$ as $x\to e$, i.e. pointwise convergence. Since $Y$ is locally compact, there exists a compact neighborhood $K$ in $Y$ containing $\g(0)$. Also $e_0^{-1}(K)$ is compact in $\G$. Therefore $l_x\g$ converges $\g$ in $\G$ as $x\to e$. 
\end{proof}
\begin{lemma}\label{lem_cont_group}
Let $\nu\in \mathcal{P}(\G)$. Then $(l_x)_\#\nu$ converges vaguely to $(l_y)_\#\nu$ as $x\to y.$
\end{lemma}
\begin{proof}
Thanks to Lemma \ref{lem_cont_ly},	without loss of generality, we take $y=e$. Let $K$ be a compact neighborhood of $e$, and
	let $\Phi\in C_c(\G)$. Observe that the family $\{\Phi\circ l_x\}_{x\in K}$ is compact in $C_c(\G)$, therefore $\Phi\circ l_x$ converges to $\Phi$ in $C_c(\G)$ as $x\to e$, which implies that
	\begin{equation}
	\lim_{x\to e}\int_\G \Phi(l_x(\g))\nu(d\g)=\int_\G\Phi(\g)\nu(d\g),
	\end{equation}
	and the thesis follows since $\Phi$ was arbitrary. 
\end{proof}
\subsection{Structural property of translation invariance}
Let us define translation invariant elements of $\mathcal{M}_\mu(\G)$. 
\begin{definition}
	We shall say that $\eeta\in \mathcal{M}_\mu(\G)$ is (left) translation invariant, if for each $y\in Y$ we have $(l_y)_\#\eeta=\eeta$. 
\end{definition}
The following structural result for translation invariant elements of $\mathcal{M}_\mu(\G)$ will be useful. 
\begin{lemma}\label{lem_transl_inv_disint}
	Let $\eeta\in \mathcal{M}_\mu(\G)$ be translation invariant. Then for every $t\in \R$, there exists a unique $\eta_{t,e}\in \mathcal{P}(\G)$ such that $\{(l_y)_\#\eta_{t,e}\}_{y\in Y}$ is a disintegration of $\eeta$ with respect to $e_t$ and $\mu$. 
\end{lemma}
Under the hypothesis of the above lemma, in view of Lemma \ref{lem_cont_group}, we shall call $\{(l_y)_\#\eta_{t,e}\}_{y\in Y}$ \emph{the} disintegration of $\eeta$ with respect to $e_t$ and $\mu$ for every $t\in \R.$
\begin{proof}
		Let $\mathscr{G}$ be a countable family of Borel sets generating the Borel $\s$-algebra of $\G$. Note that such a family exists by separability of $\G$. Let $t\in \R$ and let $\{\eta_{t,x}\}_{x\in Y}$ be a disintegration of $\eeta$ with respect to $e_t$ and $\mu$. Then we have 
	\begin{equation}
	\int_Y (l_y)_\#\eta_{t,x}\mu(dx)=(l_y)_\#\eeta=\eeta=\int_Y \eta_{t,x}\mu(dx). 
	\end{equation}
	Let $B$ be a Borel set in $Y$ and let $A\in \mathscr{G}$. We then have
	\begin{equation}
	\int_Y \eta_{t,x}(A\cap \{\g(t)\in B\})\mu(dx)=\int_Y (l_y)_\#\eta_{t,x}(A\cap \{\g(t)\in B\})\mu(dx).
	\end{equation}
	Since $\eta_{t,x}$ is concentrated on $\{\g\in\G\;:\;\g(t)=x\}$ for every $x\in Y$, we get
	\begin{equation}
	\int_{B} \eta_{t,x}(A)\mu(dx)=\int_{y^{-1}\cdot B} (l_y)_\# \eta_{t,x}(A)\mu(dx).
	\end{equation}
	Since $B$ is an arbitrary Borel set in $Y$, by Lusin's theorem there exists a $\mu$-negligeable set $N_A\subset Y$ such that for every $x\in N_A$, we have 
	\begin{equation}
	\eta_{t,x}(A)=(l_y)_\#\eta_{t,y^{-1}\cdot x}(A).
	\end{equation}
	Define now $$N:=\bigcup_{A\in\mathscr{G}}N_A,$$ which is also a $\mu$-negligeable set, and since $\mathscr{G}$ generates the Borel $\s$-algebra of $\G$, we have 
	\begin{equation}
	\eta_{t,x}=	(l_y)_\#\eta_{t,y^{-1}\cdot x} \qquad\forall x\in Y\backslash N. 
	\end{equation}
	In particular, choosing $y=x$, we get
	\begin{equation}
	\eta_{t,x}=	(l_x)_\#\eta_{t,e} \qquad\forall x\in Y\backslash N,
	\end{equation}
	whence $\{(l_x)_\#\eta_{t,e}\}_{x\in Y}$ is a disintegration of $\eeta$ with respect to $e_t$ and $\mu$, and by essential uniqueness of disintegrations $\eta_{t,e}$ is uniquely determined. 
\end{proof}

\subsection{Proof of Theorem \ref{them_main_transl_inv}}

Given a subset $I$ of $\R$, we define the subset $H^I$ of disintegrations at time $t\in I$ of translation invariant elements of $\mathcal{M}_\mu(\G)$ by 
\begin{equation}
H^I:=\{\eeta_t\in \mathcal{X}_\mu\;:\;\text{ $\eeta$ is translation invariant and $t\in I$}\}. 
\end{equation}
`\begin{lemma}\label{lem_transl_inv_comp}
	Let $I$ be a compact interval in $\R$. Then the subset $H^I$ is precompact in $\mathcal{X}_\mu$. 
\end{lemma}
\begin{proof}
	1. By Ascoli's theorem, we have to check that:
	\begin{enumerate}
		\item for each compact $A$ in $Y$, the set $H^I\lfloor_A$ of restrictions to $A$ of functions of $H^I$ is equicontinuous;
		\item for each $y\in Y$, the set $H^I(y)\subset \mathcal{P}(\G)$ of points $\eta_{t,y}$ is vaguely precompact. 
	\end{enumerate}
	2. Let us check $(i)$. Let $A$ be a compact in $Y$. Let $V$ be an entourage in the fundamental system $\mathfrak{V}$ of $\mathcal{P}(\G)$ described in Section \ref{subsec_radon_proba} of the form
	\begin{equation}
	V=\bigcap_{i=1}^N {\rm d}_{\Phi_i}^{-1}([0,\a]). 
	\end{equation}
	
	Let $t\in I$, and let $\eeta_t\in H^I$. Then we have that $(\eta_{t,x},\eta_{t,y})=((l_x)_\#\eta_{t,e},(l_y)_\#\eta_{t,e})\in V$ for every $x,y\in A$, if and only if
	\begin{equation}\label{eqn_l_x}
	\Big|\int_\G \Phi_i(l_x\g)\eta_{t,e}(d\g)-\int_\G \Phi_i(l_y\g)\eta_{t,e}(d\g)\Big|<\a \qquad\forall i=1,\dots, N.  
	\end{equation}
	In view of Lemma \ref{lem_conv_l_x} for every $\g\in \G$, we have that $l_x\g$ converges to $l_y\g$ in $\G$ as $x\to y$. Therefore, there is an open neighborhood $\mathscr{N}(y)\subset Y$ of $y$ such that for every $i=1,\dots, N$ and every $x\in \mathscr{N}(y)$ we have $\sup_{\g\in \G}|\Phi_i(l_x\g)-\Phi_i(l_y\g)|<\a$, whence \eqref{eqn_l_x} holds, i.e. $(\eta_{t,x},\eta_{t,y})\in V$. As $V$ is arbitrary in $\mathfrak{V}$, $\eeta_t$ is arbitrary in $H^I$ and $y$ is arbitrary in $Y$, it follows that $H^I\lfloor_A$ is equicontinuous.

	\bigskip 
	2.bis. Let us check $(ii)$. Let $y\in Y$ and recall that $\bigcup_{t\in I} e_{t}^{-1}(\{y\})\subset \G$ is compact. Observe now that $H^I(y)$ is contained in $\mathcal{P}(\bigcup_{t\in I} e_{t}^{-1}(\{y\}))$. By Proposition \ref{prop_comp_measure}, $\mathcal{P}(\bigcup_{t\in I} e_{t}^{-1}(\{y\}))$ is compact.
	
	This proves the thesis. 
\end{proof}

We can conclude this paper with the following proof.
\begin{proof}[Proof of Theorem \ref{them_main_transl_inv}]
	1. Let $t\in \R$. Let us show that if $\eeta$ is translation invariant, then $M_t(\eeta)$ is also translation invariant. In view of Lemma \ref{lem_transl_inv_disint}, we let $\{(l_y)_\#\eta_{t,e}\}_{y\in Y}$ be the disintegration of $\eeta$ with respect to $e_t$ and $\mu$. Then we have by Lemma \ref{lem_tensor_proj} that
	\begin{equation}
	\begin{split} 
	M_t(\eeta)&=\int_Y (l_y)_\#\eta^{(-\infty,t]}_{t,e}\otimes_{(t,y)}  (l_y)_\# \eta^{[t,+\infty)}_{t,e}\mu(dy)\\
	&=\int_Y (l_y)_\#[\eta^{(-\infty,t]}_{t,e}\otimes_{(t,e)}  \eta^{[t,+\infty)}_{t,e} ]\mu(dy). 
	\end{split} 
	\end{equation}
	So for every $x\in Y$, and every Borel set $B$ in $\G$, we have 
	\begin{equation}
	\begin{split} 
	(l_x)_\#M_t(\eeta)(B)&=\int_Y (l_x)_\#(l_{ y})_\#[\eta^{(-\infty,t]}_{t,e}\otimes_{(t,e)}  \eta^{[t,+\infty)}_{t,e} ](B)\mu(dy)\\
	&=\int_Y (l_{ y})_\#[\eta^{(-\infty,t]}_{t,e}\otimes_{(t,e)}  \eta^{[t,+\infty)}_{t,e} ](x^{-1}\cdot B)\mu(dy)\\
		&=\int_Y (l_{ y})_\#[\eta^{(-\infty,t]}_{t,e}\otimes_{(t,e)}  \eta^{[t,+\infty)}_{t,e} ]( B)(x^{-1}\cdot)_\#\mu(dy)\\
		&=\int_Y (l_{ y})_\#[\eta^{(-\infty,t]}_{t,e}\otimes_{(t,e)}  \eta^{[t,+\infty)}_{t,e} ]( B)\mu(dy)\\
		&=M_t(\eeta)(B),
	\end{split} 
	\end{equation}
	where in the second equality, we have used the definition of $l_x$, in the third equality we have used the definition of the pushforward of $\mu$ along the map $Y\ni y\longmapsto x^{-1}\cdot y\in Y,$ and in the fourth equality we have used left translation invariance of $\mu$. 
	Therefore $M_t(\eeta)$ is translation invariant and $\{(l_y)_\#[\eta^{(-\infty,t]}_{t,e}\otimes_{(t,e)}  \eta^{[t,+\infty)}_{t,e} ]\}_{y\in Y}$ is the disintegration of $M_t(\eeta)$ with respect to $e_t$ and $\mu$. 
	\bigskip 
	
2. Let $I$ be compact in $\R$, and let $F$ be a finite subset of $\R$. By 1. of this proof $M_F(\eeta)$ is translation invariant, whence by Lemma \ref{lem_transl_inv_disint} for every $t\in I$, we have that $\{(l_y)_\#\eta_{t,e;F}\}_{y\in Y}$ is the disintegration of $M_F(\eeta)$ with respect to $e_t$ and $\mu$ for some $\eta_{t,e;F} \in \mathcal{P}(\G)$, so that by Lemma \ref{lem_cont_group},  $(M_F(\eeta))_t\in \mathcal{X}_\mu$, and therefore $(M_F(\eeta))_t\in H^I$. By Lemma \ref{lem_transl_inv_comp}, the closure of $H^I$ is compact. As $I$ and $F$ were arbitrary, this shows that $\eeta$ is Markov regular. 
\end{proof}

\bibliographystyle{plain}
\bibliography{bibliografia}

\begin{bibdiv}
\begin{biblist}

\bib{ABC14}{article}{
      author={Alberti, Giovanni},
      author={Bianchini, Stefano},
      author={Crippa, Gianluca},
       title={A uniqueness result for the continuity equation in two
  dimensions},
        date={2014},
     journal={Journal of the European Mathematical Society},
      volume={16},
       pages={201\ndash 234},
}

\bib{ambrosiocrippaedi}{article}{
      author={Ambrosio, Luigi},
      author={Crippa, Gianluca},
       title={Continuity equations and {ODE} flows with non-smooth velocity},
        date={2014},
     journal={Proc. Roy. Soc. Edinburgh Sect. A},
      volume={144},
       pages={1191\ndash 1244},
}

\bib{AmbBV}{article}{
      author={Ambrosio, Luigi},
       title={Transport equation and {C}auchy problem for {B}{V} vector
  fields},
        date={2004},
     journal={Inventiones mathematicae},
      volume={158},
       pages={227\ndash 260},
}

\bib{BianchiniBonicatto20}{article}{
      author={Bianchini, Stefano},
      author={Bonicatto, Paolo},
       title={A uniqueness result for the decomposition of vector fields in
  $\mathbb{R}^d$},
        date={2020},
     journal={Inventiones Mathematicae},
      volume={220},
       pages={255 \ndash  393},
}

\bib{BianchiniDeNitti}{article}{
      author={Bianchini, Stefano},
      author={De~Nitti, Nicola},
       title={Differentiability in {M}easure of the {F}low {A}ssociated with a
  {N}early {I}ncompressible {BV} {V}ector {F}ield},
        date={2022},
     journal={Archive of Rational Mechanics and Analysis},
      volume={246},
       pages={659\ndash 734},
}

\bib{Bianchini_Marconi_Concentration}{article}{
      author={Bianchini, Stefano},
      author={Marconi, Elio},
       title={On the concentration of entropy scalar conservation laws},
        date={2016},
     journal={Discrete and Continuous Dynamical Systems Series S},
      volume={9},
       pages={73\ndash 88},
}

\bib{Bianchini_Marconi_Structure}{article}{
      author={Bianchini, Stefano},
      author={Marconi, Elio},
       title={On the {S}tructure of ${L}^\infty$-{E}ntropy {S}olutions to
  {S}calar {C}onservation {L}aws in {O}ne-{S}pace {D}imension},
        date={2017},
     journal={Archive for Rational Mechanics and Analysis},
      volume={226},
       pages={441\ndash 493},
}

\bib{BressanMazzolaNguyen-Markovian23}{article}{
      author={Bressan, Alberto},
      author={Mazzola, Marco},
      author={Nguyen, Khai~T.},
       title={Markovian solutions to discontinuous {ODE}s},
        date={2023},
     journal={Journal of Dynamics and Differential Equations},
      volume={35},
       pages={135\ndash 162},
}

\bib{bourbaki_topology}{book}{
      author={Bourbaki, Nicolas},
       title={General {T}opology},
      series={Elements of Mathematics},
   publisher={\'Editions Hermann, Paris},
        date={1966},
}

\bib{CripDeL06:estimates}{article}{
      author={Crippa, Gianluca},
      author={De~Lellis, Camillo},
       title={Estimates and regularity results for the {D}i{P}erna-{L}ions
  flow},
        date={2008},
     journal={J. Reine Angew. Math.},
      volume={616},
       pages={15\ndash 46},
}

\bib{DPL89}{article}{
      author={DiPerna, Ronald},
      author={Lions, Pierre-Louis},
       title={Ordinary differential equations, transport theory and {S}obolev
  spaces},
        date={1989},
     journal={Inventiones Mathematicae},
      volume={98},
       pages={511\ndash 547},
}

\bib{DellacherieMeyer}{book}{
      author={Dellacherie, Claude},
      author={Meyer, Paul-Andr\'e},
       title={Probabilities and potential},
      series={North-Holland Mathematics studies},
        date={1978},
}

\bib{herrlich_axiom_of_choice}{book}{
      author={Herrlich, Horst},
       title={Axiom of {C}hoice},
      series={Lecture Notes in Mathematics},
   publisher={Springer},
        date={2006},
      volume={1876},
}

\bib{Pitcho_Int}{article}{
      author={Pitcho, Jules},
       title={On the stochastic selection of integral curves of a rough vector
  field},
        date={2026},
     journal={Proceedings of the Edinburgh Mathematical Society},
}

\bib{Pitcho_SDE}{article}{
      author={Pitcho, Jules},
       title={On the zero-noise limit for {SDE}'s singular at the initial
  time},
        date={2026},
     journal={Nonlinear Differential Equations and Applications},
      volume={33},
      number={1},
}

\bib{schwartz_radon_measures}{book}{
      author={Schwartz, Laurent},
       title={{R}adon {M}easures on {A}rbitrary {T}opological {S}paces and
  {C}ylindrical {M}easures,},
      series={Tata Institute of Fundamental Research},
   publisher={Oxford University Press},
        date={1973},
}

\bib{Smirnov94}{article}{
      author={Smirnov, Stanislav},
       title={Decomposition of solenoidal vector charges into elementary
  solenoids, and the structure of normal one-dimensional flows},
        date={1994},
     journal={St. Petersburg Math. J.},
      volume={5},
       pages={841\ndash 867},
}

\bib{andre_weil_uniform}{book}{
      author={Weil, Andr\'e},
       title={Sur les espaces \`a structure uniforme et sur la topologie
  g\'en\'erale},
      series={Actualit\'es Scientifiques et Industrielles},
   publisher={\'Editions Hermann, Paris},
        date={1937},
}

\end{biblist}
\end{bibdiv}

\end{document}